%% file: repcr.tex
\documentclass[francais,twoside,openright,a4paper,draft]{amsart}

\usepackage[francais,english]{babel}     %
\usepackage[OT1]{fontenc}       %
\usepackage[utf8]{inputenc}   %
\newcommand{\comm}[1]{\medskip {\bf \tt [#1]}}
\newcommand{\commev}[1]{\medskip {\small  [#1]}}
\newcommand{\commevf}[1]{\footnote{[#1]}}
\newcommand{\compl}[1]{{\footnote{#1}}}

\renewcommand{\comm}[1]{}
\renewcommand{\commev}[1]{}
\renewcommand{\commevf}[1]{}
\renewcommand{\compl}[1]{}

\input{repcr-def.tex}

\author[A. Parreau]{Anne Parreau}
\address{Institut Fourier, UMR 5582\\BP 74\\
  Université de Grenoble I\\ 38402 Saint-Martin-d'Hères cedex, France}
\email{Anne.Parreau@ujf-grenoble.fr}
\title[Espaces de représentations complètement réductibles]
{\small 
Espaces de représentations complètement réductibles}

\begin{document}

\selectlanguage{english}

\begin{abstract}
  We study some geometric properties of actions on nonpositively
  curved spaces related to complete reducibility and semisimplicity,
  focusing on representations of a finitely generated group $\Gamma$
  in the group $G$ of rational points of a reductive group over a
  local field, acting on the associated space (symmetric space or
  affine building).
  We prove that the space of completely reducible classes is the
  maximal Hausdorff quotient space for the conjugacy action of $G$ on
  $\mathrm{Hom}(\Gamma,G)$.
\end{abstract}

\selectlanguage{francais}
\subjclass[2000]{
22E46; 53C35, 20E42, 51F99, 20E45, 14L30
}

\keywords%
{nonpositive curvature, symmetric spaces, affine buildings,  reductive
  groups over local fields, complete reducibility, moduli spaces}

\thanks{Avec le soutien de l'ANR Repsurf : ANR-06-BLAN-0311}

\maketitle

\section*{Introduction}

Soit $\Ga$ un groupe infini, engendré par une partie finie $S$.
Soit $\Es$ un espace métrique $\CAT0$ propre, muni d'une action par
isométries, propre et cocompacte, d'un groupe localement compact $G$.
Nous nous intéressons ici aux propriétés topologiques de l'action par
conjugaison (au but) de $G$ sur l'espace $\Rep=\Hom(\Ga,G)$ des
représentations de $\Gamma$ dans $G$, 
en lien avec les propriétés géométriques des représentations $\rho$ de
$\Rep$ en tant qu'actions de $\Ga$ sur $\Es$.
Notons que $\Rep$ s'identifie naturellement à un fermé de l'espace
$G^S$ des $S$-uplets de $G$ muni de l'action de $G$ par conjugaison
simultanée, via l'application $\rho\mapsto (\rho(s))_{s\in S}$.

L'espace topologique quotient $\RsG$ est en
général loin d'être séparé, ne serait-ce que parce que certaines
classes ne sont pas fermées dans $\Rep$, comme par exemple, 
lorsque $G$ est le groupe linéaire sur un corps local $\KK$,
celle d'une matrice triangulaire supérieure non
diagonale (qui adhère à sa partie diagonale). 
On peut même avoir que toutes les classes de conjugaison soient
fermées sans pour autant que le quotient soit séparé, comme par
exemple dans le cas où ($\Gamma=\ZZ$ et) $\Es$ est le plan euclidien
et $G=\Isom(\Es)$
(en effet une suite de rotations d'angles tendant vers zéro adhère
modulo conjugaison à toutes les translations).

Le cas qui nous intéresse au premier chef est celui où $G=\Gb(\KK)$,
avec $\Gb$ un groupe algébrique réductif connexe défini sur un corps
local $\KK$, agissant sur $\Es$ son espace associé (espace symétrique
dans le cas archimédien, ou immeuble de Bruhat-Tits dans le cas non
archimédien).

Dans ce cadre algébrique, pour $\KK=\RR$, la théorie des actions des
groupes algébriques réductifs réels
permet de construire un bon quotient  $\Rep//G$ à partir des orbites
fermées (\cite{Luna}, \cite{RiSl}). 
Richardson a démontré que dans ce cas les orbites fermées sont celles
des représentations {\em semisimples} (\cite{Richardson}).
La notion de représentation semisimple peut se caractériser
géométriquement par la notion suivante introduite par J.~P. ~Serre
(\cite{SerreCR}),
qui a un sens pour une action $\rho$ sur un espace métrique $\CAT0$
quelconque :
$\rho$ est {\em \cred} (cr) si, lorsque $\rho$ fixe un point $\alpha$
dans le bord à l'infini $\bordinf \Es$ de $\Es$, alors il existe un
point $\beta$ dans $\bordinf \Es$, opposé (i.e. joint par une
géodésique dans $\Es$) à $\alpha$ , également fixé par $\rho$.

Dans cet article, dans un premier temps nous étudions diverses
propriétés géo\-mé\-tri\-ques des actions sur un espace métrique
$\CAT0$ reliées à la complète réductibilité.

 Une action $\rho\in\Rep$ est dite {\em non-pa\-ra\-bo\-li\-que} si
 elle n'a pas de point fixe non trivial à l'infini de $\Es$.  Nous
 montrons (en section \ref{s- actions np sur espace CAT0}) que, dans
 un cadre $\CAT0$ très général, l'action de $G/Z(G)$ sur le
 sous-espace des représentations $\Rnp$ non-paraboliques est propre
 (où $Z(G)$ est le centre de $G$).
\`A une représentation $\rho$ nous associons une fonction convexe
$\dro:\Es\fleche \RR^+$, définie par $\dro(x)=\sqrt{\sumS
  d(x,\rho(s)x)^2}$.
Nous démontrons (section \ref{s- reps cr}, prop. \ref{prop- carac ss
  dans ES}) que, dans le cas des espaces symétriques, il y a
équivalence entre les propriétés naturelles suivantes pour une
représentation $\rho \in\Rep$.

(i) $\rho$ est \cred{}. 

(ii) $\rho$ est \np{} dans un  sous-espace convexe fermé stable de $\Es$.

(iii) La fonction  $\dro :\Es\fleche \RR^+$ atteint sa borne inférieure.

La propriété (ii) a déjà été considérée (voir par exemple
\cite{Labourie}), en lien avec des questions d'existence
d'applications harmoniques. On peut aussi noter que la condition (iii)
est équivalente à l'existence d'une application harmonique
équivariante du graphe de Cayley de $\Gamma$ vers $\Es$.

L'implication (ii)$\Rightarrow$ (iii) est valable dans un espace
$\CAT0$ propre quelconque $\Es$. L'implication (i)$\Rightarrow$ (ii)
est valable plus généralement pour $G$ réductif sur un corps local
agissant $\Es$ son espace symétrique ou immeuble associé mais dans le
cas non archimédien on n'a plus (ii)$\Rightarrow$(i) (donc plus
(iii)$\Rightarrow$(i)) (voir \ref{sss- cex npY mais pas cr}).

Dans un second temps, dans le cas algébrique général ($G$ groupe
réductif sur un corps local $\KK$ quelconque), nous donnons
(section \ref{s- separation}) une démonstration du résultat
suivant.

\begin{theonono}

 \begin{enumerate}
 \item 
\label{pt-  semisimplification}
Toute orbite de $G$ dans $\Rep$ contient dans son adhérence une unique
orbite \cred.

 \item 
\label{pt- quotient  separe}
L'espace topologique quotient $\Xcr=\Rcr/G$ de l'espace $\Rcr$ des
re\-pré\-sen\-ta\-tions cr est le plus gros quotient séparé de $\Rep$
sous $G$.
 \end{enumerate}

\end{theonono}

Dans le cas où $\KK=\RR$, ce résultat découle de \cite{Luna},
\cite{Richardson}, et \cite{RiSl}.
Pour $\KK$ de caractéristique $0$, on peut déduire les résultats
ci-dessus de \cite{Bremigan}, en utilisant que les orbites \creds{}
sont les orbites fermées (ce qui découle de \cite{Richardson} et
\cite{Bremigan}).
La démonstration que nous donnons ici est nouvelle, indépendante et
plus directe.  Elle traite de manière unifiée tous les corps locaux
sans distinction de caractéristique, dont le cas nouveau
de la caractéristique non nulle.
 Les méthodes utilisées proviennent uniquement de la géométrie en
courbure négative ou nulle et des propriétés de base des groupes
algébriques réductifs sur les corps locaux.

On utilise ce résultat dans \cite{ParComp}, où l'on construit une
compactification naturelle de $\Xcr$.

\medskip

{\bf Remerciements.} Je remercie Frédéric Paulin pour son soutien et
ses commentaires, ainsi que Michel Brion et Philippe Eyssidieux pour
des discussions instructives sur la théorie géométrique des
invariants et sur l'application moment.

\section{Notations et rappels}

Dans tout cet article, on se fixe un groupe $\Ga$ infini, de type
fini, discret, une partie gé\-né\-ra\-tri\-ce finie $S$ de $\Ga$,
et $G$ un groupe topologique métrisable, localement compact,
dénombrable à l'infini (union dénombrable de compacts), donc à base
dénombrable (d'ouverts).
Pour $g,h\in G$ on note $i_g(h)=ghg^{-1}$ la conjugaison par $g$. 
On note $Z(G)$ le centre de $G$.
On rappelle qu'une action continue de $G$ sur un espace topologique
$E$ localement compact est {\em propre} si l'application $G\times
E\fleche E\times E,\ (g,x)\mapsto (x, gx)$ est propre, ou, de manière
équivalente, si pour tous compacts $K,L$ de $E$, l'ensemble $\{g\in G,
\ gK\cap L\neq \emptyset\}$ est compact.
On note $\RGaG$ ou $\Rep$ l'espace $\Hom(\Ga,G)$ des représentations
$\rho$ de $\Ga$ dans $G$, muni de la topologie de la convergence simple.
On s'intéresse à l'action de $G$ par conjugaison (au but) sur $\Rep$
(notée $g\cdot\rho=i_g\circ\rho$). 
On note $\XGaG$ ou $\X$ l'espace topologique quotient.
L'espace $\Rep$ s'identifie à un fermé de $G^S$, par l'application
$\rho\mapsto (\rho(s))_{s\in S}$, qui est un homéomorphisme
$G$-équivariant sur son image.
En particulier, $\Rep$ est métrisable, à base dénombrable, localement
compact, dénombrable à l'infini, car $G$ l'est.
\label{TG equivalence des deux topologies naturelles}
Si $A$ est une partie $G$-stable de $\Rep$, l'ensemble quotient $A/G$
sera muni de la topologie quotient, dont on rappelle qu'elle coïncide
avec la topologie induite par l'inclusion dans $\X=\Rep/G$
\cite[Ch. III, § 2, prop. 10]{BouTG}.

\subsection{Espaces métriques $\CAT0$}
\label{ss- notations CAT0}

Dans tout cet article, l'espace $\Es$ est un espace métrique
$\CAT0$ (on renvoie par exemple à \cite{BrHa} pour la
définition et les propriétés de ces espaces) propre (c'est-à-dire dont
les boules fermées sont compactes; en particulier $\Es$ est complet et
localement compact),
muni d'une action de $G$ par isométries, continue et propre.
Une action de $\Ga$ sur $\Es$ désigne dorénavant une
action par isométries dans $G$, c'est-à-dire un élément de
$\Rep=\Hom(\Ga,G)$.
On rappelle que la propriété fondamentale des espaces métriques
$\CAT0$ est que la distance $d$ est convexe (en
restriction aux géodésiques).

\subsubsection{Bord à l'infini, sous-groupes paraboliques, faisceaux}
On note $\bordinf \Es$ le bord à l'infini de $\Es$, dont on rappelle
qu'il est formé des classes de rayons géodésiques asymptotes (i.e. à
distance bornée). 
Le stabilisateur dans $G$ d'un point $\alpha$ de $\bordinf\Es$ sera
noté $P_\alpha$ et appelé sous-groupe {\em parabolique} de $G$.
\label{ss- def points oppposes et faisceaux CAT0}
Deux points $\alpha$ et $\beta$ de $\bordinf\Es$ sont dits {\em
  opposés} s'il existe une géodésique dans $\Es$ les joignant. On note
$G_{\alpha\beta}=P_\alpha\cap P_\beta$ le sous-groupe des éléments
fixant simultanément $\alpha$ et $\beta$, et $\Es_{\alpha\beta}$ la
réunion des géodésiques de $\beta$ à $\alpha$, qui est un sous-espace
convexe fermé de $\Es$ (qu'on appelera {\em faisceau}).
On notera $\Fixinfro$ l'ensemble des points fixes d'une action 
$\rho\in \Rep$ dans le bord à l'infini de $\Es$.

\subsubsection{Facteur translaté}
\label{sss- facteur translate}
Sauf indication contraire, le produit $\Es=\Es_1\times \Es_2$ de deux
espaces métriques $\Es_1$ et $\Es_2$ sera toujours muni de la distance
produit $d_\Es=\sqrt{d_{\Es_1}^2 + d_{\Es_2}^2}$. On dira que $\Es_1$
est un facteur de $\Es$ si $\Es$ est isométrique à un produit
$\Es_1\times \Es_2$.  Le bord à l'infini $\bordinf \Es$ de $\Es$ est
alors le joint sphérique \cite[5.13]{BrHa} de $\bordinf\Es_1$ et de
$\bordinf\Es_2$ (pour la distance de Tits).
Le {\em facteur translaté (maximal)} du $G$-espace $\Es$ est le
facteur (dans une décomposition en produit préservée par $G$)
euclidien
 maximal $\Es_0$ sur lequel le groupe $G$ agit par translations.
On note alors $\Es=\Es_0\times{}\Es'$.
On a que $G$ (donc toute action) fixe point par point le bord
$\bordinf\Es_0$ du facteur translaté, et préserve $\bordinf\Es'$.

Par exemple, si $\alpha$ et $\beta$ sont deux points opposés de
$\bordinf\Es$, l'action de $G_{\alpha\beta}$ sur $\Es_{\alpha\beta}$
possède un facteur translaté non trivial naturel.
Un autre exemple d'action possédant un facteur translaté non trivial
est fourni par l'action du groupe $G=\GLnR$ sur son espace symétrique
associé $\Es=\GLn(\RR)/\Ort(n)$ (le facteur translaté maximal est ici
la droite correspondant aux orbites du centre
$Z(G)=\RR^*\mathrm{Id}$).

\subsection{Pour les groupes réductifs sur les corps locaux}
\label{s- hyp et notations groupes reductifs}
\`A partir de la section \ref{ss- cr dans les groupes reductifs}, on
se placera dans le cadre plus restreint des groupes $G$ réductifs sur
les corps locaux, agissant sur $\Es$ leur espace symétrique ou
immeuble affine associé, c'est-à-dire dans le cadre suivant.

\subsubsection{Le corps local $\KK$.}
Voir par exemple \cite[0.31]{Margulis}, \cite{BouACval}.
Soit $\KK$ un corps local, c'est-à-dire
 un corps (commutatif) localement compact non discret. 
On rappelle qu'un tel corps peut être muni d'une valeur absolue
$\abs{\cdot}$, essentiellement unique,
et que, 
si $\KK$ est archimédien, on a $\KK=\RR$ ou $\CC$,
et sinon (cas non-archimédien) $\abs{\cdot}$ est ultramétrique (et la
valuation $\omega=-\log\abs{\cdot}$ associée
 est discrète), $\KK$ est totalement discontinu, et $\KK$ est 
ou bien (en caractéristique nulle) une extension finie du corps des nombres
$p$-adiques $\QQ_p$, où $p$ est un nombre premier quelconque,
ou bien (en caractéristique positive $p$) le corps
$\FF_q((T))$ des séries formelles de Laurent  à coefficients dans un corps fini
$\FF_q$.

\subsubsection{Le groupe $G$ et l'espace métrique $\CAT 0$   associé $\Es$}
On considère  un groupe algébrique linéaire $\Gb$
connexe, réductif, défini sur  $\KK$
(par exemple $\Gb=\GLn$). 
Dans le cas où $\KK=\RR$, on suppose que $G$ est un sous-groupe fermé
de $\Gb(\RR)$ contenant sa composante neutre $\Gb(\RR)^0$.
Si $\KK\neq \RR$, on suppose  que $G=\Gb(\KK)$.
On renvoie à \cite{Borel-memo} pour la définition et un résumé des
propriétés des groupes algébriques réductifs sur un corps quelconque.

Le groupe $G$ est muni de la topologie induite par celle de $\KK$.
C'est un groupe topologique métrisable, localement compact,
dénombrable à l'infini.
Dans le cas où le corps $\KK$ est archimédien, on note $\Es$
l'espace symétrique riemannien sans facteur compact
associé à $G$. 
On renvoie à \cite{Helgason} et à \cite{Ebe} pour les propriétés
utilisés ci-dessous des groupes de Lie réels réductifs et de leurs
espaces symétriques associés (\cite[24.6]{Borel} permet de se ramener
au cas des groupes $G$ connexes).
Rappelons que $\Es=G/K$, où $K$ est un sous-groupe compact maximal de
$G$, et que cette variété est munie de la métrique riemannienne
induite par une métrique sur $G$ invariante à gauche par $G$ et à
droite par $K$.

Dans le cas non archimédien, on note $\Es$ l'immeuble de Bruhat-Tits
de $\Gb$ sur $\KK$ \cite[Sec. 2]{Tits79},
qui est un immeuble affine localement compact.
Pour la définition et les propriétés des immeubles affines (aussi
appelés euclidiens), d'un point de
vue métrique, on renvoie par exemple à \cite{ParImm}, où on pourra
aussi trouver une construction simple de l'immeuble de Bruhat-Tits de
$\Gb=\SLn$, ou à  \cite{RousseauEe04} 
(voir aussi \cite[Chap VI et Chap. V, Sec. 8]{Brown}).
Pour l'immeuble de Bruhat-Tits associé à un groupe
réductif $\Gb$ plus général  et pour ses propriétés utilisées et non
démontrées ci-dessous, on renvoie pour une bonne introduction à
\cite{RousseauEe04}, et  à \cite{Tits79}, et pour référence complète 
à \cite{BrTi}.

Les espaces symétriques et les immeubles euclidiens ont de
nombreuses propriétés en commun. 
Dans ce qui suit on a choisi de traiter autant que possible les deux
cas simultanément (les quelques références données ponctuellement
concernent le cas non archimédien, moins connu).
Le vocabulaire et les notations sont choisis par analogie avec le cas
archimédien (cas des espaces symétriques), et peut donc différer de ce qu'on
trouve usuellement dans la littérature sur les immeubles. 

Dans tous les cas, donc,  $\Es$ est un espace métrique $\CAT0$ 
propre, 
et $G$ agit continûment, 
proprement 
(pas nécessairement fidèlement),
isométriquement, 
cocompactement  sur $\Es$.
Les hypothèses de la section \ref{ss- notations CAT0} précédente sont
donc satisfaites (et on en reprendra les notations).

Notons que, dans le cas où $\Gb$ n'est pas semisimple (par exemple
pour $\Gb=\GLn$), l'espace $\Es$ possède ici un facteur euclidien
translaté par $G$ non trivial $\Es_0$, correspondant au centre de $\Gb$.

\subsubsection{Tore $A$ et plat standard $\Aa$}
Un {\em plat} d'un espace métrique $\CAT0$ désigne un sous-espace
convexe fermé de $\Es$, isométrique à un espace euclidien $\RR^m$.
Les plats maximaux (i.e. de dimension maximale) d'un immeuble affine
sont ses appartements.

On fixe une fois pour toute un plat maximal $\Aa$ de $\Es$. Il lui
correspond un tore déployé sur $\KK$ maximal $\Ab$ de $\Gb$, de telle
manière que
le tore $A=\Ab(\KK) \cap G$ de $G$ stabilise le plat $\Aa$, et agit
dessus par translations (cocompactement).
Par exemple, dans le cas où $\Gb=\GLn$, on prend le plat $\Aa$
associé au tore $\Ab=\Diagb$ des matrices diagonales.

On choisit un point $\xo$ de $\Aa$ (dans le cas non
archimédien, un point spécial).
On identifiera $\Aa$ et son espace vectoriel sous-jacent en prenant
$\xo$ comme origine.
On notera $\nu :A\fleche \Aa$ le morphisme qui à $a$
associe le vecteur de la translation correspondante. 
On peut relier relier ce vecteur de translation aux caractères sur le
tore $A$, de la manière suivante. 
On note $X^*(\Ab)=\Hom(\Ab, \mathrm{GL}_1)$
le groupe des caractères 
du tore $\Ab$.
On rappelle (cf \cite[0.2]{Tits79})
que tout caractère $\chi\in
X^*(\Ab)$ induit une forme linéaire sur $\Aa$, qu'on notera
également $\chi$,
telle que  $\chi(\nu(a))= \log\abs{\chi(a)}$ pour tout
$a$ de $A$.
Par exemple, pour $\Gb=\GLn$, si pour $i=1,\ldots, n$ on note $\eps_i
:\Ab\mapsto \GL_1$ le caractère $\Diag(a_1,\ldots , a_n)\mapsto a_i$,
on peut identifier $\Aa$ à $\RR^n$ (euclidien) de telle sorte que les
$\eps_i$, vus comme formes linéaires sur $\Aa$, soient les coordonnées
canoniques.
L'action de $a=\Diag(a_1,\ldots , a_n)\in A$ sur $\Aa=\RR^n$ est alors
la translation de vecteur $\nu(a)=(\log\abs{a_i})_{1\leq  i\leq n}$.

\subsubsection{Racines et chambre de Weyl}
Soit $\Phi\subset  X^*(\Ab)\subset \Aa^*$ le système de racines de $\Gb$
relatives au tore $\Ab$ \cite[6.3]{Borel-memo}.
Les {\em murs} de $\Aa$ sont les noyaux des racines $\varphi$ de $\Phi$.
Le {\em groupe de Weyl} $W$ est le groupe (fini) d'isométries de $\Aa$
engendré par les réflexions par rapport aux murs. 
Il fixe point par point l'intersection $\Aa_0$ des murs de $\Aa$ 
(qui est le facteur euclidien $G$-translaté maximal $\Es_0$ de $\Es$).
On choisit une {\em chambre de Weyl} $\Cc$ de $\Aa$ (dite {\em standard})
(i.e. une composante connexe du complémentaire de la réunion des murs)
et on note $\Phi^+$ l'ensemble des racines {\em positives}
(i.e. positives sur $\Cc$) et $\Lambda$ l'ensemble de racines {\em
  simples} (i.e. les racines positives qui ne se décompose pas en
somme non triviale de racines positives) correspondants.
Par exemple, dans le cas où $\Gb=\GLn$ (avec les choix ci-dessus),
l'ensemble des
racines est $\Phi=\{\varphi_{ij}=\eps_i-\eps_j,\ i\neq j\}$.
 On prend comme chambre de Weyl standard
$\Cb=\{(x_1,\ldots,x_n)\in\RR^n\ |\ x_1\geq \cdots \geq
x_n\}$.  Alors l'ensemble des racines simples est
$\Lambda=\{\varphi_{i,i+1},\ 1\leq i<n-1\}$.

\subsubsection{Sous-tores $A_I$,  plats $\Aa_I$ et facettes $\Cc_I$ standards }
Les facettes de la chambre standard $\Cc$ sont paramétrées par les
parties $I$ de $\Lambda$, de la manière suivante.
Pour une partie $I$ de $\Lambda$, on note $\Aa_I$ le sous-espace
vectoriel $\cap_{\varphi\in I} \ker\varphi$ de $\Aa$
et $\Cc_I$ la facette ouverte de $\Cc$
associée (telle que 
$v\in \Aa$ est dans
$\Cc_I$ si et seulement si $\forall\varphi\in I,\ \varphi(v)=0$
et $\forall\varphi\in\Lambda-I,\ \varphi(v)>0$).
On note $\Ab_I$ le sous-tore de $\Ab$ formé par la composante neutre
de  $\cap_{\varphi\in I}\ker\varphi$.
\label{s- A_I agit par translations cocompactes dans Aa_I}
Le groupe  $A_I=\Ab_I(\KK) \cap G$ agit sur le plat standard $\Aa$ par des
translations de vecteurs formant un sous-groupe cocompact de $\Aa_I$
(cela découle par exemple de \cite{Margulis}[2.4.2]).
On notera $v=v^I+v_I$ la décomposition d'un vecteur $v$ de $\Aa$
suivant la somme directe orthogonale $\Aa=\Aa^I \oplus \Aa_I$.

\subsubsection{Suites $I$-fondamentales}
\label{def- suite I-fond dans A}%

Une suite $(v_i)_{i\in \NN}$ dans 
$\Cb$ sera dite {\em $I$-fonda\-men\-tale} si $v_i\in \Cc_I$ pour tout
$i$ et, pour toute racine $\varphi\notin I$, on a $\varphi(v_i)\tend
\pinfty$ .
Une suite $(a_i)_i$ dans $A$ est dite {\em $I$-fonda\-men\-tale} si la
suite des vecteurs de translation $v_i=\nu(a_i)$ des $a_i$ l'est (i.e.
pour tout $\varphi \in I$, $\abs{\varphi(a_i)}= 1$ et, pour tout
$\varphi \notin I$, $\abs{\varphi(a_i)} \geq 1$ pour tout $i$ et
$\abs{\varphi(a_i)}\tend\pinfty$).

\subsubsection{Immeuble sphérique et facettes à l'infini de $\Es$}
Le bord à l'infini $\bordinf \Es$ de $\Es$ est une réalisation
géo\-mé\-tri\-que de l'immeuble sphérique combinatoire de Tits $\Delta$ de
$\Gb$ (voir \cite[11.7]{RousseauEe04}, \cite[VI, 9E]{Brown}, \cite{ParImm}).
Les facettes $f$ de cet immeuble (qui sont les bords des facettes des
chambres de Weyl de $\Es$) sont en bijection avec les
$\KK$-sous-groupes paraboliques $\Pb_f$ de $\Gb$, et on a
\label{ss- carac alg des sous-groupes parab}
$\Stab_G(f)=\Pb_f(\KK)  \cap G$ (noté $P_f$) \cite[3.1]{SerreCR}.
Le stabilisateur $P_\alpha$ dans $G$ d'un point $\alpha$ de $\bordinf
\Es$ fixe point par point (l'adhérence de) la facette ouverte
$f(\alpha)$ de cet immeuble contenant $\alpha$, en particulier
$P_\alpha=P_{f(\alpha)}$.
On appelera {\em régularité} de $\alpha$ la dimension de la facette
$f(\alpha)$.
On dit qu'une facette $f$ {\em domine} une facette $f'$ si
l'adhérence de $f$ contient $f'$.
L'{\em étoile} $\Delta_f$ (ou {\em link}, ou {\em immeuble résiduel})
de $f$, est la réunion des facettes dominant $f$.
Deux points  $\alpha$ et $\beta$ de $\bordinf\Es$ sont opposés
(cf \no\ref{ss- def points oppposes et faisceaux CAT0})
si et seulement si  
les facettes correspondantes de l'immeuble sphérique à l'infini sont
opposées,
si et seulement si  
les $\KK$-sous-groupes paraboliques de $\Gb$ correspondants sont
opposés, c'est-à-dire que leur intersection
$\Gb_{\alpha\beta}=\Pb_\alpha\cap\Pb_\beta$ est un groupe réductif,
qui est alors un sous-groupe de Levi de chacun d'eux
\cite[3.1.5]{SerreCR}.
On a  $G_{\alpha\beta}=\Gb_{\alpha\beta}(\KK) \cap G$.
Le groupe $P_\alpha$ agit transitivement sur l'ensemble des points
opposés à $\alpha$.
\subsubsection{Projection sur un facteur de Levi et groupe $U_\alpha$}
\label{ss- projection sur un levi}
On note $\Ub_\alpha$ le radical unipotent du $\KK$-sous-groupe
parabolique $\Pb_\alpha$, et $U_\alpha=\Ub_\alpha(\KK)$
 qu'on appelera le radical unipotent de $P_\alpha$.
Si $\beta\in\bordinf \Es$ est un point opposé à $\alpha$, on a une
projection naturelle $p_{\alpha\beta}$ de $P_\alpha$ sur
$G_{\alpha\beta}$. C'est un morphisme de noyau $U_\alpha$, égal à
l'identité sur $G_{\alpha\beta}$, qui correspond la décomposition en
produit semi-direct $P_\alpha=G_{\alpha\beta}U_\alpha$.
\label{ss- changement de point opp}
Si $\beta'$ est un autre point opposé à $\alpha$ dans $\bordinf\Es$, alors il
existe $g$ dans $P_\alpha$ (qu'on peut supposer dans $U_\alpha$) tel
que $\beta'=g\beta$, et les projections associées sont alors
conjuguées (on a $p_{\alpha\beta'}=i_g\circ p_{\alpha\beta}$).

\subsubsection{Sous-groupes paraboliques standards}
Soit $I\subset\Lambda$ un sous-ensemble de racines simples.  
On note $c_I$ ou $\bordinf \Cc_I$ (resp. $c_I^-$) la facette (ouverte)
à l'infini de $\Cc_I$ (resp. de $-\Cc_I$).  
On note $P_I^+$ ou
$P_I$  le sous-groupe parabolique ({\em standard}) $P_{c_I}$,
et  $\Uplus_I$ ou $U_I$ son radical unipotent.
De même, on note $P_I^-=P_{c_I^-}$ et $\Umoins_I$ son radical unipotent,
et $G_I=P_I^+ \cap P_I^-$.
Le tore $A_I$ est central dans $G_I$.

\subsubsection{Structure des faisceaux}
\label{s- structure faisceaux}
Si $\alpha$ et $\beta$ sont deux points opposés de $\bordinf\Es$, il
existe $g\in G$ et un unique $I\subset\Lambda$ tel que $g\alpha\in
c_I$ et $g\beta \in c_I^-$. Alors $gP_\alpha g^{-1}=P_I$ et $g P_\beta
g^{-1}=P_I^-$ (et $gG_{\alpha\beta} g^{-1}=G_I$, etc.).
On note $\Aa_{\alpha\beta}=g^{-1}\Aa_I$.  On a alors une
décomposition naturelle en produit
$\Es_{\alpha\beta}=\Aa_{\alpha\beta}\times\Es^{\alpha\beta}$ (qui ne
dépend pas du choix de $g$). Le groupe $G_{\alpha\beta}$ préserve
cette décomposition et agit par translations sur $\Aa_{\alpha\beta}$.
Le bord à l'infini $S_{\alpha\beta}$ de $\Aa_{\alpha\beta}$ est la
sphère de Levi (cf \cite{SerreCR}) de $\bordinf\Es$ associée au
sous-groupe de Levi $\Gb_{\alpha\beta}$ de $\Pb_\alpha$, et
$f=f(\alpha)=g^{-1}c_I$ est un simplexe maximal de $S_{\alpha\beta}$.
L'étoile $\Delta_f$ de $f$ (union des facettes fermées contenant $f$)
est incluse dans $\bordinf\Es_{\alpha\beta}$.
Le bord à l'infini de $\Es^{\alpha\beta}$
s'identifie à (une réalisation géométrique de)
$\Delta_{f}$.

\subsubsection{Décomposition $\Umoins_I{}G_I \Uplus_I{}$}
Le résultat classique suivant signifie géo\-mé\-tri\-que\-ment que, pour deux
points du bord opposés $\alpha$ et $\beta$, l'application $n\mapsto
n\beta$ est un homéomorphisme de $U_\alpha$ sur l'ensemble des points
du bord opposés à $\alpha$, qui forment un ouvert dans leur orbite.
\begin{prop}
\label{prop- decomposition NRN}
Soit $I\subset \Lambda$. L'application
      $$\fonction{\varphi}{\Umoins_I{} \times{}G_I \times{}\Uplus_I{}}{G}
      {(\umoins{},r,\upluss{})}{\umoins{} r \upluss{}}$$
est un homéomorphisme sur son image $\calO_I$, qui est un ouvert
\cite[4.2]{Borel-Tits}.\cqfd
\end{prop}

\subsubsection{Contraction par conjugaison}
Le fait classique suivant est fondamental pour l'étude de la topologie des
orbites de $G$ dans $\Rep=\Hom(\Ga,G)$. 

\begin{prop}
\label{prop- la projection est limite de conjugaisons}
Soit $I$ un sous-ensemble de racines simples.

\begin{enumerate}
\item 
La projection $p_I : P_I\fleche G_I$ est limite de conjugaisons : plus
pré\-ci\-sé\-ment, soit $a$ dans $A_I$ tel que le vecteur de
translation $\nu(a)$ de $\Aa$ associé soit dans $-\Cc_I$ 
(i.e. tel que $\forall \varphi\in  \Lambda-I,\ \abs{\varphi(a)}<1$), 
alors, pour tout
$g$ dans $P_I$, la suite $a^i g a^{-i}$ tend vers $p_I(g)$ quand l'entier $i$
tend vers l'infini. 

\item 
Soit $\rho$ dans $\Rep$ tel que $\rho(\Ga)$ est inclus dans $P_I$. La
représentation $\rho_I=p_I\circ\rho$ est dans l'adhérence de l'orbite
$A_I \cdot \rho$.
\end{enumerate}
\end{prop}

\begin{proof}
Un tel $a$ existe car $\nu(A_I)$ est  cocompact dans 
$\Aa_I$
 et  $\Cc_I$ est un cône convexe ouvert de $\Aa_I$.
Alors $a$ centralise $G_I$ et contracte $U_I$ (voir par exemple la
proposition \ref{prop- contraction par conjugaison forte} ci-dessous,
en notant que $a^{-i}$ est alors une suite $I$-fondamentale, pour une
preuve détaillée dans un cadre plus général ).
Le second point découle immédiatement du premier.
\end{proof}

La propriété de contraction plus forte suivante permet de comprendre
l'action (par conjugaison) des suites $I$-fondamentales sur les
ouverts $\Umoins_I{}G_I \Uplus_I{}$ de $G$ (voir proposition
\ref{prop- rec}). 
On note $g \cdot h = ghg^{-1}$
l'action par conjugaison de $G$ sur lui-même.

\begin{prop}
\label{prop- contraction par conjugaison forte}
Soit $(a_i)_{i\in\NN}$ une suite $I$-fondamentale  dans $A$.
Pour toute suite  $(\upluss_i)_i$
de $U_I$  (resp. $(\umoins_i)_i$ de
$U_I^{-}$) bornée, on a   $a_i^{-1}\cdot \upluss_i \tend 1$ (resp.
$a_i \cdot \umoins_i  \tend 1$) quand $i\fleche\pinfty$.
\end{prop}

\begin{proof}
Notons $\sqgb$ l'algèbre de Lie de $\Gb$ et $\sqg=\sqgb(\KK)$.
Notons $\squb_I$ l'algèbre de Lie de $\Ub_I$ et $\squ_I=\squb_I(\KK)$.
Il existe $f:U_I\fleche \squ_I$ un homéomorphisme  $A$-équivariant
(i.e. $f(a \cdot u)=\Ad(a) (f(u))$) \cite[1.3.3]{Margulis} (pour
$\KK=\RR$ c'est l'inverse de l'application exponentielle).
On note $\sqg_\varphi$ l'ensemble des $Z\in \sqg$ tels que
$\Ad(a)Z=\varphi(a)Z$ pour tout $a$ de $A$.
Alors
$\squ_I=\oplus_{\varphi\in\Phi^+_I}\sqg_\varphi$, où 
$\Phi^+_I$ est l'ensemble des racines positives qui ne sont pas
combinaison linéaire d'éléments de $I$.
Notons $Z_i=f(\upluss_i)$ et $Z_i=\sum_{\varphi\in\Phi^+_I}Z_i^\varphi$
la décomposition de $Z_i$ suivant la décomposition  
$\squ_I=\oplus_{\varphi\in\Phi^+_I}\sqg_\varphi$. 
Alors
$f(a_i^{-1}\cdot\upluss_i)=\Ad(a_i^{-1})Z_i= \sum_{\varphi\in\Phi^+_I}
\varphi(a_i)^{-1} Z_i^\varphi$. 
Or, pour tout $\varphi \in\Phi^+_I$, 
 la suite $(Z_i^\varphi)_i$ est bornée,
et $\abs{\varphi(a_i)^{-1}} \fleche 0$ 
(car $\varphi$ est combinaison linéaire positive d'éléments de $\Lambda-I$).
 On a donc que $f(a_i^{-1}\cdot\upluss_i)\tend 0$ dans $\squ_I$,
 et donc que  $a_i^{-1}\cdot\upluss_i\tend 1$ dans $U_I$.
\end{proof}

\section{Actions non-paraboliques en géométrie $\CAT0$}
\label{s- actions np sur espace CAT0}

Dans cette section, on se place dans le cadre général où 
$G$ est un groupe mé\-tri\-sable localement compact, dénombrable à
l'infini, agissant sur un espace $\Es$ métrique $\CAT0$
propre (voir section \ref{ss- notations
  CAT0} pour les notations et propriétés utilisées).

\subsection{Déplacement d'une action}
\label{ss- deplacement}

\begin{defi} 
Soit $\rho:\Ga\longrightarrow G$ une action de $\Ga$ sur $\Es$.
On appelle {\em fonction de déplacement} de $\rho$ (relativement à la
partie gé\-né\-ra\-tri\-ce $S$) et on note $\dro$ la fonction convexe
continue de $\Es$ dans $[0,\pinfty[$
$$ \dro : x  \mapsto  \sqrt{\sumS d(x,\rho(s)x)^2}.$$ 
On appelle {\em minimum de déplacement} de $\rho$ (relativement à la
partie généra\-tri\-ce $S$) et on note $\laro$ la borne inférieure de
$\dro$ sur $\Es$, et {\em ensemble de dé\-pla\-ce\-ment minimal} de
$\rho$ (relativement à la partie génératrice $S$) le convexe fermé
$\Minro$, éven\-tu\-el\-le\-ment vide, formé des points de $\Es$ où
$\dro$ atteint son minimum.
\end{defi}

\begin{remas*}
On peut voir $\dro$, $\laro$ et $\Minro$ comme une généralisation 
des notions analogues
classiques pour une isométrie individuelle $g\in \Isom(\Es)$ \cite[II,6.1]{BrHa}
(et ces notions coïncident si $S$ est réduite à un seul élément $s$
et $g=\rho(s)$).
Attention (si $S$ n'est pas réduite à un élément),  $\Minro$ n'est
a priori pas stable par $\rho(\Gamma)$.
Le centralisateur $Z(\rho)$ de $\rho$ dans $G$ préserve $\Minro$.
\label{s- semicontinuite de laro}
On notera que la fonction $\la:\Rep\fleche\RR^+$ est semicontinue 
(supérieurement) : si $\roi\tend\rho$, alors $\limsup\laroi\leq \laro$.
La définition donnée ici diffère légèrement de celle de
\cite{ParThese} (où $\dro$ est défini par $\dro(x)=\maxS d(x,\rho(s)x)$).
Cela ne change pas les propriétés qu'on y utilisait, car les deux
versions sont des fonctions convexes continues équivalentes. Celle-ci
est meilleure, notamment car elle passe bien aux produits (voir
ci-dessous), et permet 
de prouver, dans les espaces symétriques, 
 que ``$\Minro$ non vide'' entraîne la complète réductibilité
(proposition \ref{prop- carac ss dans  ES}).
\end{remas*}

{%
Dans le cas où l'espace $\Es$ est euclidien et le groupe $G$ agit par
translations, la fonction de déplacement $\dro$ est constante (égale à
$\laro$) et $\Minro$ est l'espace $\Es$ tout entier.
}
Dans le cas où $\rho$ préserve un convexe fermé $Y$, 
comme la projection sur $Y$ diminue $\dro$, on a $\la(\restrict{\rho}{Y}) =
\laro$ et $\Min(\restrict{\rho}{Y})=\Minro\cap Y$.
Dans le cas où $\rho$ préserve une décomposition de $\Es$ en produit
$\Es=\Es_1\times{}\Es_2$, on note $\rho=(\rho_1,\rho_2)$, et
on a  $\dro^2=d_{\rho_1}^2 + d_{\rho_2}^2$, d'où
$\laro^2=\la(\rho_1)^2 + \la(\rho_2)^2$ 
et $\Minro=\Min(\rho_1)\times{}\Min(\rho_2)$.

\subsection{Actions \nps{}}
\label{ss- Actions np}
Les actions (de $\Ga$ sur $\Es$) ne possédant pas de point fixe global
dans le bord à l'infini $\bordinf \Es$ de $\Es$ ont des propriétés
remarquables, que nous allons voir maintenant. Nous introduisons
maintenant la notion de représentation {\em \np{}}, qui est juste une
variante plus pratique de cette notion destinée à couvrir le cas où le
$G$-espace $\Es$ possède un facteur translaté non trivial (maximal) $\Es_0$
(cf \ref{sss- facteur translate}). On note alors $\Es=\Es_0\times{}\Es'$.

\begin{defi} 
\label{def- action np} 
On dit qu'une action $\rho$ de $\Ga$ sur $\Es$ est {\em \np{}}
si $\rho$ n'a 
pas de point fixe global dans $\bordinf \Es -
\bordinf \Es_0$ (ou, de manière équivalente, dans $\bordinf \Es'$).
\end{defi}
Remarquons qu'une représentation est \np{} si et seulement si les
éléments de la partie génératrice $S$ n'ont pas de point fixe commun
dans $\bordinf \Es'$.

\subsubsection{Liens avec l'irréductibilité 
et la stabilité pour les groupes réductifs.}
\label{s- liens entre np et irreducibilite et stabilite dans les GR}
Dans le cas où $G$ est $\GLnR$ agissant sur son espace symétrique
$\Es$ associé, une re\-pré\-sen\-ta\-tion est \np{} si et seulement si
l'action linéaire sur $\RR^n$ associée est ir\-ré\-duc\-ti\-ble.
Dans le cas où $G$ est un groupe réductif sur un corps local $\KK$,
agissant sur son espace associé $\Es$ (espace symétrique ou immeuble
de Bruhat-Tits) (cf. hypothèses et notations de la section \ref{s- hyp
  et notations groupes reductifs}),
une représentation $\rho$ est \np{} si et seulement si elle est
$\Gb$-{\em irréductible} au sens de \cite[3.2.1]{SerreCR},
c'est-à-dire que son image $\rho(\Ga)$ n'est incluse dans aucun
$\KK$-sous-groupe parabolique propre de $\Gb$ 
(en particulier les représentations (d'images) \zdenses{}
sont \nps{}).
Cette notion est donc légèrement plus générale  que la notion de
repré\-sen\-ta\-tion (ou $S$-uplet) {\em stable} de la théorie
géométrique des invariants (i.e. $\Gb\cdot\rho$ fermé de Zariski et
$Z_\Gb(\rho)/Z(\Gb)$ fini),
qui équivaut  à la condition que $\rho(\Ga)$ n'est inclus dans
aucun sous-groupe parabolique (pas nécessairement défini sur $\KK$) de
$\Gb$ \cite[4.1 et 16.7]{Richardson} (voir aussi par exemple
\cite{JoMi}).
\subsubsection{Caractérisation via l'ensemble minimal}
\label{s- dro decroit vers les pf a l'infini}
Si $\alpha$ est un point fixe à l'infini de $\rho$, alors la fonction
de déplacement $\dro$ est décroissante sur tout rayon géodésique
allant vers $\alpha$ (par convexité).
Cela donne la caractérisation fondamentale suivante.
\begin{prop} \label{prop- spfi ssi minro compact non vide}
Soit $\rho$ une action de $\Ga$ sur $\Es$. Alors sont équivalents :
\begin{enumerate}
\item L'action $\rho$ n'a pas de point fixe à l'infini.  

\item La fonction de déplacement $\dro:\Es\fleche\RR^+$ est propre.

\item L'ensemble minimal $\Minro$ est  compact non vide.  

\end{enumerate}
En particulier, si $G$ agit proprement sur $\Es$, le centralisateur
$Z(\rho)$ de $\rho$ dans $G$ est alors compact.
\cqfd 
\end{prop}

\begin{rema}
\label{rem- np implique Minro non vide}
On a que $\rho$ est \np{} si et seulement si l'action $\rho'$ sur
$\Es'$ 
induite est sans points fixes à l'infini, si et seulement si
$\Min(\rho')$ est compact non vide. Alors en particulier
$\Minro=\Es_0\times \Min(\rho')$ est non vide, 
et, si $G$ agit proprement sur $\Es$ et $Z(G)$ agit cocompactement sur
$\Es_0$, on a que $Z(\rho)/Z(G)$ est compact.
\end{rema}

\subsection{Action de $G$ sur l'espace des actions \nps{}}
On note maintenant $\Rspfi(\Ga, G)$ ou $\Rspfi$ le sous-espace de
$\Rep$ formé par les représentations de $\Ga$ dans $G$ sans point fixe
à l'infini, c'est-à-dire n'ayant pas de point fixe global
dans $\bordinf \Es$.
On note $\Rnp(\Ga, G)$ ou $\Rnp$ le sous-espace de $\Rep$ formé par
les représentations \nps{} de $\Ga$ dans $G$. 
La propriété de base suivante
permet de relier la ``non propreté'' de l'action de $G$  au
voisinage de $\rho\in\Rep$ avec les points fixes à l'infini de
$\rho$.

\begin{lemm}
\label{lemm- existence de points fixes au bord}
On considère une suite $(\rho_i)_i$ convergeant vers $\rho$ dans $\Rep$, et
 une suite $(g_i)_i$ dans $G$ telle que la suite des conjugués $\rho'_i
= g_i \cdot \rho_i$ reste dans un compact de $\Rep$.
 Si pour un (tout) point $x$ de $\Es$, la suite des points $g_i^{-1}x$ tend
vers un point $\alpha$ dans le bord à l'infini de $\Es$, alors $\rho$
fixe $\alpha$.
\end{lemm}

\begin{proof}
Pour $\ga $ dans $\Ga$ fixé, le  déplacement par 
$\rho_i(\ga)$ du point $x_i=g_i^{-1}x$ vaut $d(x_i,\ \rho_i(\ga) x_i) 
= d( x,\ \rho'_i(\ga) x)$, donc  est borné.
Par conséquent $\rho_i(\ga) x_i $ tend vers $ \alpha$.
Comme  $\rho_i(\ga)x_i$ tend aussi vers $ \rho(\ga)\alpha$,
on a que $\rho(\ga)\alpha = \alpha$.
\end{proof}

Le résultat suivant s'en déduit directement.

\begin{prop} \label{prop- Rspfi ouvert, lambda continue, Xspfi lc} 
\begin{enumerate}

\item L'espace $\Rspfi$ est un ouvert de $\Rep$.

\item Si $G$ agit proprement sur $\Es$, alors $G$ agit proprement sur $\Rspfi$.

\end{enumerate}
\end{prop}

\begin{proof}
La première assertion est claire par compacité du bord à l'infini de
$\Es$
 et continuité de l'action de $G$ (une limite de points fixes est
un point fixe pour l'action limite).
Montrons que $G$ agit proprement sur $\Rspfi$.
Commme $G$ et $\Rep$ sont métrisables, il suffit de voir que, si on a
une suite $(\rho_i)_i$ qui tend vers $\rho$ dans $\Rspfi$ et une suite
$(g_i)_i$ dans $G$ telle que $g_i \cdot \rho_i$ tend vers $\rho'$ dans
$\Rspfi$, alors la suite $(g_i)_i$ possède une valeur d'adhérence dans
$G$. Or, dans le cas contraire, pour $\xo\in\Es$, la suite
$g_i^{-1}\xo$ n'est pas bornée dans $\Es$ (car $G$ agit proprement sur
$\Es$), donc, quitte à extraire, $g_i^{-1}\xo$ tend vers $\alpha$ dans
le bord de $\Es$,
qui est fixé par $\rho$ par le lemme \ref{lemm- existence de points
  fixes au bord}.  C'est impossible car $\rho$ n'a pas de point fixe à
l'infini.
\end{proof}

Le résultat analogue suivant permet de couvrir le cas où le $G$-espace
$\Es$ possède un facteur translaté non trivial (l'espace $\Rspfi$ est
alors vide), en particulier il s'applique au cas des groupes réductifs
non semisimples (cf. section \ref{s- hyp et notations groupes
  reductifs}).

\begin{coro} \label{coro- Rnp ouvert, Xnp lc} 
On suppose que le centre $Z$ de $G$ agit trivialement sur le facteur
$\Es'$.
On note $G'=G/Z$, qui agit sur $\Es'$.
 On suppose que l'action de $G'$ sur $\Es'$ est propre.
\begin{enumerate}
\item L'espace $\Rnp$ est un ouvert de $\Rep$. 

\item Le groupe $G'$ agit proprement sur $\Rnp$.
\end{enumerate}

En particulier $\Rnp/G$ est séparé et localement compact.
\end{coro}

\begin{remas*}
1. Si l'action de $G$ sur $\Es$ est cocompacte, minimale (sans
sous-espace convexe strict stable) (ce qui est automatique si les
géodésiques de $\Es$ sont extensibles \cite[II,6.20]{BrHa}), alors $Z$
agit trivialement sur $\Es'$ (car $Z$ agit par translations de
Clifford sur $\Es$ \cite[II,6.16]{BrHa}, de directions fixées par $G$,
i.e. dans $\Es_0$).
Si de plus l'action de $G$ sur $\Es$ est propre, et $Z$ agit cocompactement sur
$\Es_0$, alors $G/Z$ agit proprement sur $\Es'$.
Les hypothèses sont donc vérifiées dans le cas des groupes réductifs
agissant sur leur espace associé (cf. section \ref{s- hyp et notations
  groupes reductifs}).

2. Ce résultat généralise donc le résultat suivant de \cite[proposition
  11.11]{Richardson} :
lorsque  $G$ est un groupe de Lie ré\-duc\-tif ré\-el, 
l'action de $G/Z$ sur le sous-espace $(G^n)^s$ des $n$-uplets stables
est propre.

3.
En particulier, sous les hypothèses ci-dessus, on a donc que si $\rho$
est \np{}, alors $G\cdot\rho$ est fermé dans $\Rep$ et $Z(\rho)/Z(G)$
est compact. 
Dans les groupes réductifs on peut montrer que cette propriété
caractérise en fait les représentations \nps{} (en utilisant le
théorème \ref{thm- semisimplification et PGQS} et le corollaire
\ref{coro- Serre-cr implique np dans Y faisceau}) 
(à comparer avec la notion de repré\-sen\-ta\-tion stable, où compact
est remplacé par fini, cf \ref{s- liens entre np et irreducibilite et
  stabilite dans les GR}).

\end{remas*}

\begin{proof}
L'action de $Z$ sur $\Rep(\Ga,G)$ est triviale, donc l'action de $G$
passe au quotient en une action continue de $G'$ sur $\Rep(\Ga,G)$. 
On note $\mu$ la projection canonique de $G$ sur $G'$, et $\mu^*$
l'application continue de $\Rep(\Ga,G)$ dans $\Rep(\Ga,G')$ induite, qui
est $G'$-équivariante.
Par la proposition \ref{prop- Rspfi ouvert, lambda continue, Xspfi
  lc}, on a que $\Rspfi(\Ga, G')$ est un ouvert, donc $\Rnp(\Ga,
G)=(\mu^*)^{-1}(\Rspfi(\Ga, G'))$
est aussi un ouvert. De plus $G'$ agit proprement sur $\Rspfi(\Ga,
G')$, donc aussi sur $\Rnp(\Ga, G)$ \cite[III.29, prop. 5]{BouTG}.
\end{proof}

\section{Représentations \creds}
\label{s- reps cr}

 Dans cette section, on introduit la notion de représentation \cred{}
 et on étudie des propriétés géométriques (i.e. de l'action sur
 l'espace métrique $\CAT0$), naturellement reliées, dont on montre
 qu'elles sont en fait équivalentes dans le cas des espaces
 symétriques (prop. \ref{prop- carac ss dans ES}).

\subsection{Dans le cadre général des espaces métriques $\CAT0$}
\label{ss- reps cr CAT0}

 \begin{defi}
\label{def- ss : Serre-cr}
On dira qu'une  représentation $\rho$ (de $\Ga$ dans $G$) est
 {\em \cred}  (en abrégé, cr) (dans $\Es$) si elle satisfait la
condition suivante :
Si un point $\alpha$ dans le bord à l'infini $\bordinf \Es$ de $\Es$
est fixé par $\rho$, alors il existe un point $\beta$ dans $\bordinf
\Es$, opposé à $\alpha$, également fixé par $\rho$.
 \end{defi}

En particulier, une représentation \np{} est \cred{}.
Cette notion a été introduite et étudiée par J.~P.~Serre
(\cite{SerreCR}) dans le cadre des groupes agissant sur des immeubles
sphériques (par exemple les sous-groupes des groupes algébriques
réductifs sur un corps quelconque).

\subsubsection{Liens avec la semisimplicité dans les groupes réductifs}
\label{s- liens entre cr et semisimplicite dans les GR}
Dans le cas où $G$ est un groupe réductif sur un corps local agissant
sur son espace $\CAT0$ associé (notations et hypothèses de la section \ref{s- hyp et notations groupes reductifs}),
une représentation $\rho:\Ga\fleche G$ est \cred{} si et seulement si
$\rho(\Ga)$ est $\Gb$-cr au sens de \cite[3.2.1]{SerreCR}.
Dans le cas où $G=\GLn\KK$, une représentation est
\cred{} si et seulement si l'action linéaire sur $\KK^n$ associée est
semi-simple.

Si la caractéristique du corps $\KK$ est nulle, alors $\rho$ est
\cred{} si et seulement si  la composante neutre $\Hb=(\ovz{\rho(\Ga)})^0$ de
l'adhérence de Zariski de $\rho(\Ga)$ est un groupe réductif
(\cite[Proposition 4.2]{SerreCR}).
Cela correspond donc à $\rho$ semisimple (comme
$S$-uplet de $\Gb$) au sens de \cite{Richardson}.
Pour $\KK=\RR$, Richardson a introduit une notion naturelle (dépendant
seulement de la structure de groupe de Lie réel de $G$) de semisimplicité
pour les $S$-uplets dans $G^S$ (l'algèbre de Lie $\sqg$ de $G$ est
un module semisimple sous le groupe engendré), qui est
équivalente à la précédente (\cite[section 11]{Richardson}), donc à la
complète réductibilité.
En caractéristique quelconque, dans un corps algébriquement clos,
Bate, Martin et Röhrle ont démontré \cite{BMR} que la notion de
complète réductibilité est équivalente à la notion de ``forte
réductivité'' de Richardson \cite{Richardson}.

\subsubsection{Sous-espace stable et ensemble minimal}
Une propriété d'une action $\rho$ naturellement reliée à la
``réductibilité'' est la suivante (correspondant à la définition
``$\mathrm{Im}\rho$ réductif'' de \cite{Labourie} dans le cadre où
$\Es$ une variété riemannienne simplement connexe à courbure négative
ou nulle) :
{\em $\rho$  stabilise un sous-espace convexe fermé $Y$ de $\Es$ et est \np{}
dans $Y$.}
On a alors la propriété suivante.

\begin{prop}[Ensemble minimal non vide]
\label{prop- np dans Y implique Minro non vide}
On suppose que
la re\-pré\-sen\-ta\-tion $\rho$ 
est \np{} dans un sous-espace (convexe fermé, stable) $Y$ de $\Es$.
Alors
la fonction de déplacement $\dro$ 
atteint sa borne inférieure, autrement dit $\Minro$ est non vide.
\end{prop}

\begin{proof}
  En effet, on a $\Min(\restrict{\rho}{Y})=\Minro\cap Y$ par
  projection orthogonale sur le convexe fermé $Y$. Cet ensemble est
  non vide car $\rho$ est \np{} dans $Y$ (voir la remarque \ref{rem-
    np implique Minro non vide}).
\end{proof}

\subsubsection{Ensemble minimal non vide et topologie des orbites}
Voici quelques ``bonnes'' propriétés élémentaires des actions $\rho$
telles que $\Minro$ est non vide, relativement à la topologie des
orbites sous conjugaison.

\begin{prop}
\label{prop- semisimplification par Minro}
Supposons que $\Es$ est quasi-homogène sous $G$, i.~e.~que l'action de $G$
sur $\Es$ est cocompacte.  
\begin{enumerate}
\item Pour toute re\-pré\-sen\-ta\-tion $\rho$ de
$\Rep$, il existe une représentation $\rho'$ dans l'adhérence 
de $G\cdot\rho$ telle que $\Min(\rho')$ est non vide, et
$\la(\rho')=\laro$. 
En particulier, si l'orbite $G \cdot \rho$ de $\rho$ est fermée, alors
$\Minro$ est non vide.

\item 
\label{item- conjuguer vers laroi mini}
Soit $(\roi)_i$ 
telle que $\laroi$
  est borné.  Alors il existe une sous-suite de con\-ju\-gués
  $\rho_{i_k}'\in G \cdot \rho_{i_k}$ convergeant vers $\rho$ telle que
  $\Min(\rho)$ n'est pas vide, et $\la(\rho)=\liminf\laroi$.
\end{enumerate}
\end{prop}

\begin{proof}
Le premier point est un cas particulier du deuxième point, prouvons ce
dernier.
Soit $\la=\liminf\laroi$. Quitte à extraire on peut supposer que
$\laroi\tend \la$.
Pour tout entier $i$, soit $x_i$ un point de $\Es$ tel que $\laroi\leq
\droi(x_i)\leq \laroi +\frac{1}{i}$.  Quitte à remplacer chaque $\roi$
par un conjugué, on peut supposer que la suite $(x_i)$ est bornée (car
l'action de $G$ sur $\Es$ est cocompacte).  Quitte à extraire, on a
alors que $x_i$ tend vers $ x$ et $\roi$ tend vers $ \rho$.  Pour tout
$y\in\Es$, on a que $\droi(y)\geq\laroi$ pour tout $i$, 
donc que $\dro(y)\geq \lim\laroi=\la$. Or $\dro(x)=\la$.  On en
conclut que $\laro=\liminf\laroi$ et que $x$ est dans $\Minro$.
\end{proof}

\subsubsection{Symétrie en un point intérieur}
Dans le cas des espaces symétriques, on verra que la propriété
``$\Minro$ non vide'' suffit en fait à caractériser les
représentations \creds{} (proposition \ref{prop- carac ss dans
  ES}) (ce qui n'est pas le cas dans les arbres par exemple, voir le
contre-exemple ci-dessous). 
Nous allons maintenant en voir l'idée principale. 

Dans le cas où $\Es$ est à géodésiques uniquement extensibles
(i.e. tout rayon géodésique se prolonge de manière unique en une
géodésique complète) (par exemple une variété de Hadamard), on peut
définir l'{\em opposé en $x\in\Es$} de $\alpha\in\bordinf\Es$. 
On dit qu'une partie $Y$ de $\bordinf \Es$ est {\em symétrique par
  rapport à $x$}, si, pour tout point $\alpha$ de $Y$, l'opposé
$\beta$ de $\alpha$ en $x$ est aussi dans $Y$.
On dira que $\Es$ est {\em sans demi-bande plate}, si les rayons
parallèles se prolongent en géodésiques parallèles (c'est le cas des
espaces symétriques ou plus largement des variétés de Hadamard
analytiques) (notons qu'un analogue de cette condition est utilisé
dans \cite{Labourie}).
On a alors la propriété suivante.

\begin{prop}[$\Minro$ non vide implique cr]
\label{prop- Min(rho) non vide implique x symetrie des pts fixes a l'infini}
On suppose que $\Es$ est à géodésiques uniquement extensibles et sans
demi-bande plate. On suppose que $\Minro$ est non vide.
Alors,
pour tout point $x$ dans $\Minro$, l'ensemble $\Fixinfro$ des
points fixes de $\rho$ dans $\bordinf\Es$ est symétrique
par rapport à $x$.
En particulier, $\rho$ est cr.
\end{prop}

\begin{proof}
Supposons $\Minro$ non vide. Soit $x$ un point de $\Minro$. Montrons
que $\Fixinfro$ est symétrique par rapport à $x$.
Soit $\alpha$ un point fixe à l'infini de $\rho$ et 
$r$ la géodésique issue de $x$ vers $\alpha$. 
Pour tout $s$ de $S$, la fonction $t \mapsto d(r(t),\ \rho(s)r(t))$
est décroissante sur $\RR$.
De plus la somme $\dro(r(t))=\sqrt{\sumS d(r(t),\ \rho(s)r(t))^2}$ atteint son
minimum en $0$.
Par conséquent, pour tout $s$ de $S$, la fonction $t\mapsto d(r(t),\
\rho(s)r(t))$ est constante sur $\RR^+$,
donc sur $\RR$,
c'est-à-dire  que $\rho$ fixe le sy\-mé\-tri\-que $r(\minfty)$ en $x$ de
$\alpha=r(\pinfty)$.
\end{proof}

\begin{rema*}
On a au passage démontré que, pour $\Es$ quelconque, si $x$ est
dans $\Minro$ alors le cône en $x$ sur $\Fixinfro$ est inclus dans
$\Minro$. Il en découle que, si $\Minro$ est non vide alors
$\bordinf\Minro = \Fixinfro$.
\end{rema*}

\subsubsection{Un contre-exemple} 
\label{sss- cex npY mais pas cr}
Dans le cas général, 
``\np{} dans un sous-espace convexe fermé stable''
n'entraîne pas que $\rho$ soit cr (ou que l'orbite $G\cdot\rho$ soit
fermée), comme le montre le contre-exemple simple suivant.

On considère le groupe $G=\SL_2\KK\times\SL_2\KK$, pour un corps local
$\KK$ non archimédien, agissant sur son immeuble de Bruhat-Tits $\Es$
associé (produit de deux arbres $\Es_1$ et $\Es_2$).  Soit $t\in\KK$
tel que $\abs{t}>1$ et $g=(g_1,g_2)$ dans $G$, avec $g_1=
\begin{pmatrix}
t&0\\
0&t^{-1}
\end{pmatrix}
$
et $g_2=
\begin{pmatrix}
1&t\\
0&1
\end{pmatrix}
$. 
Alors $g_1$
translate une géodésique $\sigma$ de $\Es_1$ et $g_2$ fixe un rayon
géodésique $r$ de $\Es_2$
donc $g$ translate une géodésique $Y=\sigma\times r(0)$  de $\Es$. 
En particulier $g$ est  \np{} dans le sous-espace convexe fermé $Y$.
Néanmoins, $g$ n'est pas \cred{} dans $\Es$, car le point à l'infini
$r(\pinfty)$ n'a pas de point opposé fixé par $g$.  Et $G \cdot g$
n'est pas fermée dans $G$, car son adhérence contient $(g_1,Id)$.

\subsection{Dans les groupes réductifs}
\label{ss- cr dans les groupes reductifs}

On suppose désormais que $G$ est un groupe réductif sur un corps local
$\KK$, agissant sur son espace $\CAT0$ associé $\Es$ (espace
sy\-mé\-tri\-que ou immeuble affine) (hypothèses et notations de la section
\ref{s- hyp et notations groupes reductifs}).
Les résultats de cette section découlent pour la plupart, par une
simple traduction dans $\Es$, de \cite{SerreCR}, sauf le point
\ref{item- la reduction conserve la} de la proposition \ref{prop-
  reduction d'une rep parabolique}.
\begin{prop} (\cite[Prop 2.7]{SerreCR})
  Si une représentation $\rho$ fixe deux points opposés $\alpha,\beta$
  à l'infini de $\Es$, alors les deux conditions suivantes sont équivalentes

(i) L'action de $\rho$ sur $\Es$ est \cred.

(ii) L'action de $\rho$ sur le sous-espace stable $\Es_{\alpha\beta}$ est \cred.
\cqfd\end{prop}

On s'intéresse maintenant au dévissage des représentations $\rho$
paraboliques via leurs projections sur des sous-groupes de Levi 
(voir \ref{ss- projection sur un levi}) (qu'on
appelera {\em réductions} de $\rho$), qui permet de ``semisimplifier''
$\rho$, comme le montre la proposition suivante.

\begin{prop}
\label{prop- reduction d'une rep parabolique}
Soit $\rho$ dans $\Rep$. Soit $\alpha$ un point fixe de $\rho$ dans
$\bordinf\Es$ et $\beta$ un point de $\bordinf \Es$ opposé à
$\alpha$. Soit $\rho_{\alpha\beta} = p_{\alpha\beta}\circ \rho$.
\begin{enumerate}

\item 
\label{item- ss implique toute projection lui est conjuguee}%
Si $\rho$ est \cred{}, 
alors elle est conjuguée à 
$\rho_{\alpha\beta}$ (par un élément de $U_\alpha$).

\item 
\label{item- la reduction conserve la}
On a $d_{\rho_{\alpha\beta}}\leq \dro$ sur $\Es_{\alpha\beta}$ 
et
  $\la(\rho_{\alpha\beta})= \laro$.

\item (voir aussi \cite[prop. 3.3]{SerreCR})
 \label{item- si rho_alpha beta fixe f' alors rho fixe f''}
Si  $\rho_{\alpha\beta}$ fixe un point $\alpha'$ dans 
$\bordinf\Es^{\alpha\beta}$, alors 
$\rho$ fixe une facette $f''$ de $\bordinf\Es$
dominant strictement la facette $f(\alpha)$.

\item (voir aussi \cite[prop. 3.3]{SerreCR})
Si $\alpha$ est un point de régularité maximale dans $\Fixinfro$,
alors $\rho_{\alpha\beta}$ est
\np{} dans le faisceau $\Es_{\alpha\beta}$.

\end{enumerate}
\end{prop}

\begin{proof}
Le point \ref{item- ss implique toute projection lui est
  conjuguee} découle du fait que, en prenant  $\beta'$
opposé à $\alpha$ fixé par $\rho$,  les projections correspondantes
$p_{\alpha\beta}$ et $p_{\alpha\beta'}$ (qui fixe $\rho$) sont
conjuguées (cf. section (\ref{ss- changement de point opp})).  

Pour le point \ref{item- la reduction conserve la}, 
quitte à conjuguer on peut se ramener au cas standard où $\alpha\in
c_I$  et $\beta \in c_I^-$ pour un certain $I\subset\Lambda$ (cf
section \ref{s- structure faisceaux}).
La projection $p_{\alpha\beta}=p_I$ est la limite de conjugaisons par
$a^i$ avec $a\in A_I$ tel que $\nu(a)\in -\Cc_I$ (cf prop. \ref{prop-
  la projection est limite de conjugaisons}).
Soit $x\in\Es_{\alpha\beta}$, et $s\in S$.  Soit $\sigma$ la
géodésique translatée par $a^{-1}$ passant par $x$. Comme $g=\rho(s)$ fixe la
facette ouverte contenant $\alpha$, qui est $c_I$, on a que $g$ fixe
aussi $\sigma(\pinfty)$ (car $\nu(a^{-1})\in\Cc_I$), donc que la
distance entre les géodésiques $\sigma$ et $g\sigma$ est
décroissante. Donc $d(x, a^iga^{-i}x)=d(a^{-i}x, ga^{-i}x)\leq d(x,
gx)$, et $d(x, p_{\alpha\beta}(g)x)\leq d(x, gx)$ par passage à la limite en
$i\tend\pinfty$. Donc $d_{\rho_{\alpha\beta}}(x)\leq \dro(x)$.

Pour la seconde partie,  comme
$\rho_{\alpha\beta}$ est limite de conjugués de $\rho$, on a que
$\laro\leq \la(\rho_{\alpha\beta})$ par semicontinuité de $\la$.
Montrons que $\la(\rho_{\alpha\beta})\leq \laro$.
Soit $\eps>0$ et $y\in\Es$ tel que $\dro(y)\leq \laro +\eps$.
Soit $\beta'$ opposé à $\alpha$ tel que $y\in\Es_{\alpha\beta'}$.
Par ce qui précède on a $\la(\rho_{\alpha\beta'})\leq
d_{\rho_{\alpha\beta'}}(y)\leq \dro(y)\leq \laro+\eps$. Or
$\la(\rho_{\alpha\beta'})=\la(\rho_{\alpha\beta})$ car
$p_{\alpha\beta}$ et $p_{\alpha\beta'}$, donc $\rho_{\alpha\beta}$ et
$\rho_{\alpha\beta'}$, sont conjuguées.  On conclut en faisant tendre
$\eps$ vers $0$.

Pour le point \ref{item- si rho_alpha beta fixe f' alors rho fixe
  f''}, 
comme le groupe $U_\alpha=\ker p_{\alpha\beta}$ fixe point par point
l'étoile $\Delta_f$ de $f=f(\alpha)$, la représentation $\rho$ coincïde avec
$\rho_{\alpha\beta}$ sur $\Delta_f$.
Or $\rho_{\alpha\beta}$ fixe le segment de $\alpha$ à $\alpha'$ (pour
la distance de Tits), donc la facette $f'$ de l'étoile de $f(\alpha)$
qui contient dans son intérieur un germe de ce segment.
Le dernier point est une conséquence directe du précédent.
\end{proof}

On en déduit immédiatement la caractérisation suivante.

\begin{coro}
\label{coro- Serre-cr implique np dans Y faisceau}
\label{coro- np dans Y faisceau implique Serre-cr}
Pour $\rho$ dans $\Rep$,
les conditions  suivantes sont équivalentes.   
 \begin{enumerate}
\item $\rho$ est \cred{}.

\item 
Ou bien  $\rho$ est \np{} dans $\Es$, ou bien
il existe  $\alpha$ et $\beta$ dans $\bordinf\Es$ opposés, fixés
par $\rho$, tels que $\rho$ est \np{} dans le faisceau $\Es_{\alpha,\beta}$.
\cqfd  \end{enumerate}
En particulier $\Minro$ est alors non vide.
\end{coro}

\subsection{Cas des espaces symétriques}

On a maintenant obtenu l'équivalence des différentes caractérisations
géométriques remarquables suivantes, dans les espaces symétriques
(i.e. dans le cas des groupes réductifs réels, cf section \ref{s- hyp
  et notations groupes reductifs}).

\begin{prop}
\label{prop- carac ss dans ES}
On suppose que $G$ est un groupe réductif réel, agissant
sur son espace symétrique sans facteur compact $\Es$ associé.
Soit $\rho:\Ga\fleche G$ une représentation.
les conditions  suivantes sont équivalentes.   

\begin{enumerate}
\item
\label{item- Serre-cr a l'infini}
$\rho$ est \cred{} dans $\Es$.
\item 
\label{item- np dans Y Faisceau}
Ou bien $\rho$ est \np{} dans $\Es$, ou bien
$\rho$ fixe deux points à l'infini opposés $\alpha$ et $\beta$,
et est \np{} dans le faisceau $\Es_{\alpha,\beta}$.

\item \label{item- np dans Y}
$\rho$ stabilise un sous-espace convexe fermé $Y$ de $\Es$,
et est \np{} dans $Y$.

\item \label{item- Min(rho) non vide}
La fonction de déplacement $\dro: x\mapsto \sqrt{\sumS d(x,\rho(s)x)^2}$ 
atteint sa borne inférieure, autrement dit $\Minro$ est non vide.

\item 
\label{item- x symetrie des pts fixes a l'infini}
Il existe un point $x$ de $\Es$ tel que $\Fixinfro$  est symétrique
par rapport à $x$.
\end{enumerate}

\end{prop}

\begin{proof}
On vient de voir que (\ref{item- Serre-cr a l'infini}) équivaut à
(\ref{item- np dans Y Faisceau}) dans le cadre des groupes réductifs
(\cite{SerreCR}, cf. Corollaire \ref{coro- Serre-cr implique np dans Y
  faisceau}).
On a clairement que (\ref{item- np dans Y Faisceau}) implique
(\ref{item- np dans Y}), et que (\ref{item- x symetrie des pts fixes a
  l'infini}) implique (\ref{item- Serre-cr a l'infini}).
L'implication {(\ref{item- np dans Y}) 
$\Rightarrow$ (\ref{item- Min(rho) non vide})} 
est vraie dans tout espace métrique $\CAT0$ propre
$\Es$ (proposition \ref{prop- np dans Y implique Minro non vide}).
L'implication  {(\ref{item- Min(rho) non vide})
$\Rightarrow$ (\ref{item- x symetrie des pts fixes a l'infini})} 
est vraie, pour tout $x$ de $\Minro$, dans le cadre plus général où
$\Es$ est à géodésiques uniquement extensibles et sans demi-bande
plate (proposition \ref{prop- Min(rho) non vide implique x symetrie des pts
  fixes a l'infini}).
\end{proof}

\begin{rema*} On peut démontrer que  l'ensemble minimal $\Minro$ est ici
  un sous-espace totalement géodésique (car si $\dro$ est
  constante sur un segment géodésique $\sigma[0,1]$, alors on montre
  aisément que $d_{\rho(s)}$ aussi pour tout $s$ de $S$) sur lequel le
  centralisateur $Z(\rho)$ de $\rho$ agit transitivement (car le
  transport parallèle le long de $\sigma$ - prolongée à $\RR$ - est
  dans $Z(\rho)$). On peut alors en déduire les propriétés
  suivantes, analogues aux  résultats de base de la théorie de
  l'application moment (voir par exemple \cite{RiSl}) : si on note
  $\mathcal{M}$ le sous-ensemble des $\rho\in\Rep$ telles que $\dro$
  atteint sa borne inférieure en $\xo$, alors $\rho$ est cr si et
  seulement si $G\cdot \rho$ rencontre $\mathcal{M}$, et on a alors
  $\mathcal{M}\cap G\cdot\rho=K\cdot \rho$, où $K$ est le sous-groupe
  compact $\Stab_G(\xo)$.
\end{rema*}
\section{Séparation et espace quotient}

\label{s- separation}

On se place dorénavant dans le cadre où $G$ est un groupe réductif sur
un corps local agissant sur son espace (espace symétrique ou immeuble
affine) associé (notations et hypothèses de la section \ref{s- hyp et
  notations groupes reductifs}).
Dans cette section on montre (théorème \ref{thm- semisimplification et
  PGQS}) que l'espace $\Xcr$ des classes de représentations cr est le
plus gros quotient séparé de $\Rep=\Hom(\Ga,G)\subset G^S$ sous
l'action (par conjugaison) de $G$.

On note $\Rcr$ le sous-espace formé des représentations \creds, et
$\Xcr=\Rcr/G$ l'espace topologique quotient.
 Notons que l'espace $\Xcr$ est à base dénombrable (car $\Rep$ l'est).
On dira que deux points $x$ et $x'$ d'un espace topologique $E$ sont
{\em séparés} par une action continue de $G$ dans $E$ s'il existe deux
voisinages $V$ et $V'$ de $x$ et $x'$ dont les orbites $G \cdot V$ et
$G \cdot V'$
ne se rencontrent pas. Sinon, on dira que $x$ et $x'$ sont {\em
$G$-voisins dans $E$} (on notera que ce n'est pas {\em a priori}
 une relation d'équivalence).

\subsection{Séparation des orbites cr}
\label{ss- separation}
Nous allons ici démontrer le ré\-sul\-tat de sé\-pa\-ra\-tion des orbites cr
suivant, qui est une variante (contenant l'essentiel) du théorème
\ref{thm- semisimplification et PGQS} (et donc connu en caractéristique
nulle, voir les remarques suivant le théorème \ref{thm-
  semisimplification et PGQS}).
On rappelle (cf. section \ref{ss- cr dans les groupes reductifs}) 
qu'une {\em réduction} d'une représentation $\rho$ désigne une
représentation de la forme $\sigma=p_{\alpha\beta}\circ \rho$ où $p_{\alpha\beta} :P_\alpha\fleche G_{\alpha\beta}$ est la
projection canonique associée à deux points à l'infini opposés
$\alpha$ et $\beta$ avec $\alpha$ fixé par $\rho$. 

\begin{theo}
\label{theo- Rcr sur G separe}
\label{theo- orbites voisines}
Soient $\rho$ et $\rho'$ deux points $G$-voisins dans $\Rep$.
Alors 

\begin{enumerate}

\item $\rho$ et $\rho'$ possèdent deux réductions 
conjuguées dans  $G$.
\item Les adhérences de leurs orbites se rencontrent ($\ov{G \cdot
    \rho}\cap \ov{G \cdot \rho'} \neq \es$).
\item Si de plus  $\rho$ et $\rho'$ sont cr,
  alors elles sont conjuguées (${G \cdot \rho} = {G \cdot
    \rho'}$).
 \end{enumerate}

En particulier, l'espace topologique quotient $\Xcr=\Rcr/G$ est séparé.  
\end{theo}

Pour démontrer ce théorème, on commence par  le petit lemme suivant,
qui montre que, comme ``les
parties compactes ne comptent pas'', on peut se ramener, grâce à un
analogue de la dé\-com\-po\-si\-tion de Cartan, à l'action d'une suite
$I$-fondamentale de $A$.

\begin{lemm}
\label{lemm- reduction a une suite fondamentale}
On considère une action continue de $G$ sur un espace topologique $E$ à base
dénombrable. Soient $x, x'$ dans $E$.
Les assertions suivantes sont équivalentes.
\begin{enumerate}
\item Les points $x$ et $x'$  sont $G$-voisins dans $E$.

\item il existe $h,k$ dans $G$, 
une partie  $I$ de $\Lambda$, 
une suite $I$-fondamentale $(a_i)_i$ de $A$, 
et une suite $(y_i)_i$ dans $E$, tels que $y_i \tend y$ et 
$a_i \cdot y_i \tend y'$ avec $y= h \cdot x$ et $y'= k \cdot x'$ quand $i\tend \pinfty$.

\end{enumerate}
\end{lemm}

\begin{proof}
Si $x$ et $x'$ sont $G$-voisins, il existe une suite $(x_i)_i$ dans $E$
 et une suite $(g_i)_i$ dans $G$ telles que 
$x_i \tend x$ et $g_i \cdot x_i \tend x'$ quand $i\tend \pinfty$.
Comme $\Cb$ est un domaine fondamental pour l'action du sous-groupe
compact $K=\Stab_G(\xo)$ sur $\Es$,
il existe $k_i\in K$ tel que  $g_i \xo=k_i v_i$ avec $v_i$ dans $\Cb$.
Soit $I$ l'ensemble des racines simples $\varphi$ de $\Lambda$
non bornées sur la suite $(v_i)$.  La suite des projections $v^I_i$ de $v_i$
sur $\Aa^I$ est  alors bornée.  La suite des projections $v_{i,I}$ de $v_i$
sur $\Aa_I$ est à distance bornée d'une suite $u_i=a_i  \xo$, avec
$a_i\in A_I$ (car $A_I$ agit cocompactement sur $\Aa_I$). 
On note $g_i=k_ia_ih_i$.
La suite $(h_i)_i$ est bornée dans $G$ car la suite
$h_i\xo=a_i^{-1}v_i$ du plat $\Aa$ a ses composantes dans
$\Aa^I\oplus\Aa_I$ bornées.
On a donc, quitte à extraire, que $k_i$ tend vers $k^{-1}$ et $h_i$
tend vers $h$ dans $G$, et la suite $y_i = h_i \cdot x_i$ convient.
L'autre sens est clair.
\end{proof}

On  décrit maintenant les limites de conjugués par une suite
$I$-fondamentale $(a_i)_i$ de $A$.
La preuve repose sur la décomposition $\Umoins_I{}G_I \Uplus_I{}$ (cf
prop. \ref{prop- decomposition NRN}).

\begin{prop}[Action d'une suite $I$-fondamentale dans $A$.]
\label{prop- conjugaison par une suite I-fondamentale dans A}
Soit $I$ une partie de $\Lambda$ et $(a_i)_i$ une suite $I$-fondamentale  dans $A$.
Soient deux représentations $\rho$ et $\rho'$ de $\Rep$, et une suite
$(\rho_i)_i$ convergeant vers $ \rho$ dans $\Rep$, 
telles que la suite des conjugués $(a_i \cdot \rho_i)_i$ tend vers $ \rho'$. 
Alors

\begin{enumerate}
 
\item 
\label{i- paraboliques}
La représentation $\rho'$ est (d'image) incluse dans le
  sous-groupe parabolique $P_I^+$, 
et $\rho$ est (d'image) incluse dans
 le sous-groupe parabolique opposé $P_I^-$.

\item \label{i- r=r'}
Les projections $r$ et $r'$ de $\rho$ et $\rho'$ sur le
  sous-groupe de Levi commun $G_I$ sont
  égales. 

\item  
\label{item- detail}
Soit $\rho=\umoins{} r$ la décomposition de $\rho$ suivant la
décomposition $P_I^-=\Umoins_I G_I$ et $\rho'=r \upluss{}'$ la
décomposition de $\rho'$ suivant la décomposition $P_I^+ = G_I
\Uplus_I$.  Alors, à partir d'un certain rang, on a
$\rho_i=\umoins{}_i r_i \upluss{}_i $, pour des suites $r_i\tend r$ dans
$(G_I)^\Ga$, et  $\umoins_i \tend \umoins$
  dans $(\Umoins_I)^\Ga$, et enfin   $\upluss{}_i =a_i^{-1} \cdot
\upluss{}'_i $ où $\upluss'_i \tend \upluss'$ dans
  $(\Uplus_I)^\Ga$.
\end{enumerate}
\end{prop}

Remarquons que le point \ref{item- detail} décrit complètement la
situation. En effet, les propriétés de contraction de la conjugaison
par la suite $(a_i)$ sur $\Umoins_I$ et $\Uplus_I$ (proposition
\ref{prop- contraction par conjugaison forte}) permettent de voir qu'on a
la réciproque suivante 
(attention, ici $\rho_i :\Ga\fleche G$ n'a  pas de raisons a priori
d'être un morphisme).

\begin{prop}
\label{prop- rec}
Soient $\rho=\umoins{} r$ une représentation à valeurs dans $P_I^- =
\Umoins_I G_I$ et $\rho' = r \upluss{}'$ une représentation à valeurs
dans $P_I^+ = G_I \Uplus_I$, de même projection $r$ sur $G_I$.  
Pour toute suite $I$-fondamentale
$(a_i)$ de $A_I$, si $\rho_i=\umoins_i r_i \upluss_i $ avec $r_i\tend
r$ dans $(G_I)^\Ga$ et $\umoins_i \tend \umoins$ dans
$(\Umoins_I)^\Ga$ quelconques, et $\upluss_i =a_i^{-1} \cdot \upluss'_i
$ avec $\upluss'_i \tend \upluss'$ dans $(\Uplus_I)^\Ga$ quelconque, alors on
a
$\rho_i\tend \rho$  et $a_i \cdot \rho_i\tend \rho'$.
En particulier, $\rho$ et $\rho'$ sont $A$-voisines.
\cqfd
\end{prop}

\begin{proof}[Preuve de la proposition \ref{prop- conjugaison par une suite
      I-fondamentale dans A}]
Il suffit de  prouver ces points ``terme à terme'', c'est-à-dire  pour un élément
  fixé $\gamma$ de $\Gamma$. On note alors $g=\rho(\gamma)$ et
  $g'=\rho'(\gamma)$.
On commence par projeter, en utilisant une décomposition
$\UmoinsJ{}G_J\UplusJ{}$, dans un $G_J$, a priori plus gros que $G_I$,
tel que  $P_J^+$ contient $g'$ et $P_J^-$ contient  $g$.
  Puis on montre qu'en fait $J=I$ par un argument de minimalité, d'où
  le résultat pour le bon $I$.

{\em  Existence d'une facette  à l'infini fixée ``commune''.}
On note $v_i=\nu(a_i)$ (vecteur  de la translation de $\Aa$ 
réalisée par $a_i$).
Quitte à extraire, on peut supposer que $v_i$ tend vers un point $v$
dans le bord à l'infini de la facette $\Cb_I$. Soit $\Cc_L$ la facette
ouverte de $\Cc_I$ contenant $v$. Alors $g'$ fixe la facette à
l'infini $\bordinf \Cc_L$ (d'après le lemme \ref{lemm- existence de
  points fixes au bord}).
De même, en remplaçant $v_i$ par $-v_i$ qui tend vers $-v$, on voit
que $g$ fixe le bord à l'infini de la facette opposée $\Cc_L^-=-\Cc_L$.

{\em Choix de $J$ minimal.}
On suppose désormais que $\Cc_J$ est une facette maximale (pour l'inclusion)
parmi les facettes de $\Cc_I$  telle que  $g'$ fixe $\bordinf\Cc_J$ (c'est-à-dire  
$g'\in P_J^+$) et  $g$ fixe $\bordinf(\Cc_J^-)$  (c'est-à-dire 
$g\in P_J^-$). 

Soit $g=\umoins{} r$ la décomposition de $g$ suivant la
décomposition $P_J^-=\Umoins_J G_J$ et $g'=r \upluss{}'$ la
décomposition de $g'$ suivant la décomposition $P_J^+ = G_J
\Uplus_J$.

{\em Projection sur $G_J$ via la décomposition $ \UmoinsJ{} G_J \UplusJ{}$.}
On con\-si\-dè\-re maintenant la décomposition $\calO_J=\UmoinsJ{} G_J
\UplusJ{}$ de la proposition \ref{prop- decomposition NRN}.
Comme $\calO_J$ est un ouvert qui contient $P_J^-$ et $P_J^+$, donc $g$ et $g'$,
les suites $g_i$ et $g'_i$ sont dans $\calO_J$ à partir d'un
certain rang.
On a donc les décompositions
$g_i=\umoins{}_i r_i  \upluss{}_i $ et
$g'_i =\umoins{}_i'  r'_i  \upluss{}_i' $, 
où $\umoins{}_i$ et  $\umoins{}_i'$ sont dans $\UmoinsJ{}$,
$\upluss{}_i$ et $\upluss{}_i'$ sont dans $\UplusJ{}$, 
et $r_i$ et $r'_i$ sont dans $G_J$.
La décomposition $\UmoinsJ{} G_J
\UplusJ{}$ étant unique, continue, et conservée par la conjugaison par
$a_i$, on a que $r_i\tend r$ et que
$r'_i =a_i\cdot r_i\tend r'$.
De même, on a que $\umoins_i\tend\umoins$ et que
$\upluss_i'=a_i\cdot\upluss_i\tend\upluss'$.

{\em Montrons enfin que $I=J$.}
Notons $v_i=v_i^J+v_{i,J}$ la décomposition de $v_i$ suivant la somme
orthogonale $\Aa=\Aa^J \oplus \Aa_J$.
Si $J\neq I$, alors $v_i^J$ n'est pas borné : 
en effet il existe $\varphi\in J-I$, et $\varphi(v_i^J)=\varphi(v_i)$
tend vers $\pinfty$, car la suite $v_i$ est $I$-fondamentale.
Quitte à extraire, on a donc que $v_i^J$ tend vers un point $v^J$ dans le bord à
l'infini  de $\Aa^J$.
Comme $A_J$ agit cocompactement sur $\Aa_J$, il existe $b_i\in A_J$
tel que $v_{i,J}-\nu(b_i)$ reste borné.
Notons $a_i'=a_ib_i^{-1}$, alors $a'_i\xo$ tend encore vers $v$.
Comme $A_J$ est central dans $G_J$, on a que $a_i'.r_i = a_i.r_i$, qui
converge vers $r'$.
On en déduit que $r'$ fixe $v^J$ 
(par le lemme \ref{lemm- existence de points fixes au
  bord}).
Comme $\varphi(v^J)\geq 0$ pour
toutes les racines $\varphi$ de $J$ et qu'on a $\varphi(v^J)>0$ pour au
moins une racine $\varphi$ de $J$ (car $v^J\notin\Aa_J$), on a
que le cône $\RR^+v^J\oplus\Aa_J$
rencontre une facette ouverte $\Cc_{J'}$ dominant strictement la
facette $\Cc_J$.
Comme $r'$ fixe tous les points du bord à l'infini de $\RR^+v^J+\Aa_J$,
il fixe la facette à l'infini $\bordinf\Cc_{J'}$. Donc $g'$ fixe
également $\bordinf\Cc_{J'}$ (car $U_J$  fixe
$\bordinf\Cb$).
On voit de même, en remplaçant $a_i'$ par $(a_i')^{-1}$
que $g$ fixe le bord à l'infini de la facette opposée $\Cc_{J'}^-$ (car
$(a_i')^{-1}\xo \tend -v^J$).
Ce qui contredit l'hypothèse ``$\Cc_J$ maximale'' faite ci-dessus,
car $\Cc_{J'}$ domine strictement la facette $\Cc_J$.
On a  donc en fait  $r'_i = r_i$ pour tout $i$ (car $a_i\in A_J$
central dans $G_J$), d'où $r'=r$.
\end{proof}

\begin{proof}[Preuve du théorème \ref{theo- Rcr sur G separe}]
  Comme les réductions d'une représentation sont dans l'ad\-hé\-ren\-ce de
  son orbite (cf
  proposition \ref{prop- la projection est limite de conjugaisons}) et
  qu'une représentation \cred{} est conjuguée à toutes ses réductions
  (cf proposition \ref{prop- reduction d'une rep parabolique}, point
  \ref{item- ss implique toute projection lui est conjuguee}), il
  suffit de voir le premier point.
  Quitte à conjuguer $\rho$ et $\rho'$, on peut supposer qu'il existe
  une suite $I$-fondamentale $(a_i)$ dans $A$ et une suite $\rho_i$
  dans $\Rep$ convergeant vers $\rho$ telles que $a_i \cdot \rho_i$
  converge vers $\rho'$ (lemme \ref{lemm- reduction a une suite
    fondamentale}). Les points (\ref{i- paraboliques}) et (\ref{i-
    r=r'}) de la proposition \ref{prop- conjugaison par une
    suite I-fondamentale dans A} permettent alors de conclure.
\end{proof}

\subsection{Semisimplification et plus gros quotient séparé}
Soient $\rho$ et $\rho'$ dans $\Rep$. 
D'après le théorème \ref{theo- orbites voisines}, $\rho$ et $\rho'$ sont $G$-voisines si et seulement si  
$\ov{G \cdot \rho}\cap \ov{G \cdot \rho'}\neq \es$. On note dans ce cas
$\rho\sim\rho'$, et il est facile de voir que $\sim$ est alors une
relation d'équivalence. On note $\Rep//G=\Rep/\sim$ l'espace topologique
quotient, et $\psep:\Rep\fleche \Rep//G$ la projection correspondante (qui
passe au quotient en une application continue surjective 
$\ov \psep : \Rep/G \fleche \Rep//G$). 

\begin{theo}
\label{thm- semisimplification et PGQS}

\begin{enumerate}

\item 
\label{item- semisimplification}
Pour tout $\rho$ de $\Rep$, l'adhé\-rence de  $G \cdot
\rho$ contient une unique orbite \cred.

\medskip

On note $\pi :\Rep \fleche \Xcr$ la projection $G$-invariante associée
({\em semisimplification}).

\item
\label{item- plus gros quotient separe}
L'application $\pi$ est continue
et induit un homéomorphisme de $\Rep//G$ sur $\Xcr$.
En particulier, $\Rep//G$ (resp. $\Xcr$) est le plus gros quotient séparé de
$\Rep$ sous $G$  (toute application continue $G$-invariante $f$ de
$\Rep$ vers un espace séparé factorise à travers $\psep$ (resp. $\pi$)).

\end{enumerate}
\end{theo}

\begin{remas*}
Le point \ref{item- semisimplification} est prouvé dans \cite{Richardson}
pour $\KK=\RR$.
Par ailleurs, un résultat analogue, mais en remplaçant les orbites cr
par les orbites fermées, et pour des actions plus générales de
groupes réductifs, est prouvé, pour
$\KK=\RR$, dans \cite{Luna} et \cite{RiSl}, 
et plus généralement, pour $\KK$ de
caractéristique nulle; dans \cite[5.4 et 5.11]{Bremigan}.
Il implique alors le résultat ci-dessus quand en utilisant que les orbites
fermées sont exactement les orbites \creds{} (\cite{Richardson} pour
$\KK=\RR$, et, plus généralement pour $\KK$ de caractéristique
nulle, on peut le déduire de \cite{Richardson}
et \cite{Bremigan}).
\end{remas*}

\begin{proof}
Pour le point \ref{item- semisimplification} :
  En projetant sur un sous-groupe de Levi correspondant à un point
  fixe de $\rho$ dans $\bordinf\Es$ de régularité maximale, on obtient
  bien une représentation \cred{} adhérente à l'orbite de $\rho$ par
  la proposition \ref{prop- reduction d'une rep parabolique} et le
  corollaire \ref{coro- np dans Y faisceau implique Serre-cr}.
L'unicité à conjugaison près découle de la séparation des orbites cr
(théorème \ref{theo- Rcr sur G separe}).

Pour le point \ref{item- plus gros quotient separe},
il s'agit essentiellement de montrer la continuité de $\pi$. 
Les arguments suivants sont inspirées par des idées de Maxime Wolff 
\cite[2.2.6]{Wolff}.
Comme $\Rep$ et $\Xcr$ sont  à base dénombrable et séparés  (théorème
\ref{theo- Rcr sur G separe}), il suffit de montrer que, si $\roi$
tend vers $\rho$ dans $\Rep$, alors, quitte à extraire, $\pi(\roi)$
tend vers $ \pi(\rho)$.
On peut tout d'abord supposer $\rho$ \cred{} (quitte à remplacer
$\rho$ par $\sigma$ cr dans $\ov{G \cdot \rho}$, et $\roi$ par
une suite extraite et conjuguée qui converge vers $\sigma$).
Si $\roi$ n'a pas de point fixe à l'infini, et est donc \cred{}, à
partir d'un certain rang, on a que $\pi(\roi)=G\cdot\roi$ tend vers
$G\cdot\rho=\pi(\rho)$ dans $\Rcr/G$, ce qui conclut.
Sinon, quitte à extraire, il existe pour tout $i$ un point fixe à
l'infini $\alpha_i$ de $\roi$, qu'on peut choisir de régularité
maximale et de type constant (c'est-à-dire dans une orbite de $G$
fixée).
Quitte à extraire, on peut supposer que $\alpha_i$ tend vers un point
$\alpha$ de $\bordinf \Es$, qui est alors fixé par $\rho$.
Quitte à remplacer $\roi$ par un conjugué $k_i \cdot \roi$ avec
$(k_i)_i$ une suite dans le sous-groupe compact $K=\Stab_G(\xo)$
convergeant vers $1_G$, on peut supposer que $\alpha_i$ est toujours
égal à $\alpha$ (en prenant $k_i$ tel que $k_i\alpha_i=\alpha$).
Comme $\rho$ est \cred{}, elle fixe un point $\beta$ de $\bordinf \Es$
opposé à $\alpha$.
La composée $\sigma_i$ de $\roi$ par la projection $p_{\alpha\beta}$
est alors une semisimplification de $\roi$ (proposition \ref{prop-
  reduction d'une rep parabolique} et corollaire \ref{coro- np dans Y
  faisceau implique Serre-cr}), et tend vers
$p_{\alpha\beta}\circ\rho=\rho$.
On a donc que $G\cdot\sigma_i=\pi(\roi)$ tend vers
$G\cdot\rho=\pi(\rho)$ dans $\Rcr/G$, ce qui conclut pour la
continuité de $\pi$.

L'application $\pi$ passe alors au quotient en $\ov\pi:\Rep//G\fleche
\Xcr$ continue, et on vérifie aisément que la restriction de $\ov
\psep: \Rep/G\fleche \Rep//G$ à $\Xcr$ en est un inverse.
\end{proof}

On inclut pour finir quelques propriétés de
$\Rep//G$ ayant un intérêt propre.

\begin{prop}
L'espace $\Xcr$ est localement compact et dénombrable à l'infini.
\end{prop}

\begin{proof}
En effet l'image par $\pi:\Rep\fleche\Xcr$ de la trace sur $\Rcr$
d'une base dénombrable d'ouverts relativement compacts de $\Rep$
fournit une base dénombrable d'ouverts (car la restriction de $\pi$ à
$\Rcr$ est ouverte et surjective)
relativement compacts (car $\pi$ est continue sur $\Rep$) de $\Xcr$.
\end{proof}

\begin{prop}
L'application ``minimum de déplacement'' $\la:\Rep\fleche\RR^+$ est
continue, et passe au quotient en une fonction continue et propre sur
$\Rep//G$.
\end{prop}

\begin{rema}
Dans un espace $\CAT0$ quelconque, $\la$ est semicontinue, mais elle
n'est pas continue en général (par exemple, dans le cas du plan euclidien).
\end{rema}

\begin{proof}
Montrons tout d'abord que $\la$ passe au quotient. 
On a $\la(\pi(\rho))=\laro$ par la proposition \ref{prop- reduction
  d'une rep parabolique} (\ref{item- la reduction conserve la})(car
$\pi(\rho)$ est une réduction de $\rho$).  Donc $\la$ passe au
quotient en une fonction $\Rep//G\fleche \RR^+$ qu'on notera aussi
$\la$.

Montrons maintenant la continuité de $\la$ sur $\Rep$ : 
soit $(\roi)_{i\in\NN}$ une suite dans $\Rep$ telle que $\roi\tend \rho$.
Alors $\la(\roi)$ est borné. Supposons (quitte à extraire) 
que $\laroi\tend\ell$.
On peut alors choisir une suite de conjugués $\rho'_i$ de $\rho_i$
telle que $\rho'_i\tend\rho'$ avec $\la(\rho')=\ell$
(prop. \ref{prop- semisimplification par Minro}, (\ref{item- conjuguer
  vers laroi mini})).
On a que $\pi(\rho'_i)\tend \pi(\rho')$ et $\pi(\rho_i)\tend \pi(\rho)$
par continuité de $\pi$ (thm \ref{thm- semisimplification et PGQS}), or
$\pi(\rho'_i)=\pi(\roi)$ pour tout $i$, 
donc $\pi(\rho)=\pi(\rho')$ (car $\Xcr$ séparé).  On a vu qu'alors
$\laro=\la(\rho')=\ell$, ce qui conclut.

L'application quotient $\la:\Rep//G\fleche \RR^+$ est donc continue.
Montrons qu'elle est propre.
Soit $D >0$. Comme $G$ agit cocompactement sur $\Es$, il existe
$C\in\RR$ tel que si $\laro < D$ alors il existe $g\in G$ tel que
$d_{g\cdot \rho}(\xo)\leq C$.
En particulier $\psep(\rho)$ est alors dans l'image par $\psep$
continue du compact $\{\rho\in\Rep,\ \dro(\xo)\leq C\}$ de $\Rep$, qui
est un compact de $\Rep//G$, ce qui conclut.
\end{proof}

\input repcr-bib.tex
\end{document}

%% file: repcr-def.tex
\usepackage{amssymb,amsmath}           
\usepackage{amsfonts} %
\usepackage{bm}

\usepackage{amsthm} %

\newtheorem{prop}{Proposition}

\newtheorem{theo}[prop]{Th{\'e}or{\`e}me}

\newtheorem{lemm}[prop]{Lemme}
\newtheorem{coro}[prop]{Corollaire}

\newtheorem*{theonono}{Th{\'e}or{\`e}me}

\theoremstyle{definition}
\newtheorem{defi}[prop]{D{\'e}finition}

\theoremstyle{remark}
\newtheorem{rema}[prop]{Remarque}

\newtheorem*{rema*}{Remarque}
\newtheorem*{exem*}{Exemple}
\newtheorem*{remas*}{Remarques}
\newtheorem*{exems*}{Exemples}

\newcommand{\RR}{\mathbb{R}}
\newcommand{\NN}{\mathbb{N}}
\newcommand{\ZZ}{\mathbb{Z}}
\newcommand{\CC}{\mathbb{C}}
\newcommand{\QQ}{\mathbb{Q}}
\newcommand{\KK}{\mathbb{K}}

\newcommand{\FF}{\mathbb{F}}

\newcommand{\calO}{{\mathcal O}}

\newcommand{\eps}{\varepsilon}        %

\newcommand{\goth}[1]{\mathfrak{#1}}
\newcommand{\boldgoth}[1]{\bm{\goth{#1}}}

\newcommand{\cqfd}{\qed}

\newcommand{\ov}{\overline}

\newcommand{\fleche}{\longrightarrow}   %
\newcommand{\tend}{\rightarrow}%

\newcommand{\bord}{\partial}    %

 \newcommand{\restrict}[2]{{#1}_{|#2}}    %

\newcommand{\fonction}[5]{%
{#1} \, :
\begin{array}{ccc}
{#2} & \rightarrow & {#3} \\
{#4} & \mapsto & {#5}
\end{array}   
}

\newcommand{\es}{\emptyset}           %
\newcommand{\abs}[1]{\left| #1 \right|}     %

\newcommand{\pinfty}{{+\infty}}         %
\newcommand{\minfty}{{-\infty}}         %

\newcommand{\Min}{\mathrm{Min}}         %

\newcommand{\limi}%
         {\mathop{\underline{\mathrm{lim}}}\nolimits} %
\newcommand{\lims}%
         {\mathop{\overline{\mathrm{lim}}}\nolimits}  %

 \newcommand{\Hom}{\mathrm{Hom}} %
 \newcommand{\Stab}{\mathrm{Stab}}        
 \newcommand{\Fix}{\mathrm{Fix}}

 \newcommand {\Isom} {{\mathrm{Isom}\:}} %

\newcommand{\Diag}{\mathrm{Diag}}        
\newcommand{\GL}{\mathrm{GL}}        
        
\newcommand{\GLn}{\GL_n}

\newcommand{\SL}{\mathrm{SL}}        
        
\newcommand{\SLn}{\SL_n}

\newcommand{\GLnR}{\GLn(\RR)}

\newcommand{\Ort}{\mathrm{O}}

         \newcommand{\Ad}{\mathrm{Ad}}

 \newcommand{\sqg}{{\goth g}}
 \newcommand{\squ}{{\goth u}}
 \newcommand{\squb}{\boldgoth{u}} %
 \newcommand{\sqgb}{\boldgoth{g}}

 \newcommand{\umoins}{u}
 \newcommand{\upluss}{n}
 \newcommand{\Umoins}{U^-}
 \newcommand{\Uplus}{U^+}
 \newcommand{\UmoinsJ}{\Umoins_J}
 \newcommand{\UplusJ}{\Uplus_J}

 \newcommand{\Fixinfro}{\Fix_\infty(\rho)}

          \newcommand{\np}{non-pa\-ra\-bo\-li\-que}    %
          \newcommand{\nps}{\np s}    %
         \newcommand{\cred}{com\-plè\-te\-ment ré\-duc\-ti\-ble}    
         \newcommand{\creds}{\cred s}    

         \newcommand{\CAT}[1]{\mathrm{CAT}(#1)}

          \newcommand{\la}{\lambda}

         \newcommand{\ga}{\gamma} %
         \newcommand{\Ga}{\Gamma} %

         \newcommand{\maxS}{\max_{s\in S}} %
         \newcommand{\sumS}{\sum_{s\in S}}

         \newcommand{\bordinf}{\bord_\infty}
        \newcommand{\Aa}{\mathbb{A}} %
        \newcommand{\Cc}{{\goth C}} %
          \newcommand{\Cb}{\overline{\goth C}} %

          \newcommand{\Gb}{{\bf G}}
          \newcommand{\Hb}{{\bf H}}
          \newcommand{\Pb}{{\bf P}}
          \newcommand{\Ab}{{\bf A}}
          \newcommand{\Ub}{{\bf U}}
\newcommand{\Diagb}{{\bf \mathrm{Diag}}}

          \newcommand{\zdenses}{Zariski-denses}
          \newcommand{\ovz}[1]{\overline{#1}^Z}

         \newcommand{\Es}{X} %
         \newcommand{\xo}{{x_0}}

         \newcommand{\Rep}{\mathrm{R}} %
         \newcommand{\X}{{\mathcal X}}   %
        
         \newcommand{\psep}{p}  %
         \newcommand{\Xcr}{\X_{cr}} %
         \newcommand{\XGaG}{\X(\Ga,G)}   %

         \newcommand{\RsG}{\Rep/_G}      %
         \newcommand{\Rnp}{\Rep_{np}}    %
         \newcommand{\Rcr}{\Rep_{cr}}    %
         \newcommand{\Rspfi}{\Rep_{0}}    %
         \newcommand{\RGaG}{\Rep(\Ga,G)} %

         \newcommand{\dro}{d_{\rho}}
         \newcommand{\laro}{\la(\rho)}
         \newcommand{\Minro}{\Min(\rho)}

         \newcommand{\roi}{{\rho_i}}
         \newcommand{\droi}{d_{\roi}}
         \newcommand{\laroi}{\la(\roi)}